\documentclass[12pt]{article}
\usepackage[latin9]{inputenc}
\usepackage[letterpaper]{geometry}
\usepackage{color}
\usepackage{dsfont}
\usepackage{amsmath}
\usepackage{amsthm}
\usepackage{amssymb}
\usepackage{graphicx}
\usepackage[unicode=true,pdfusetitle,
 bookmarks=true,bookmarksnumbered=false,bookmarksopen=false,
 breaklinks=false,pdfborder={0 0 1},backref=page,colorlinks=true]
 {hyperref}
\hypersetup{
 linkcolor=blue, citecolor=blue}

\makeatletter
\numberwithin{equation}{section}
\numberwithin{figure}{section}
\theoremstyle{plain}
\newtheorem{thm}{\protect\theoremname}[section]
\theoremstyle{plain}
\newtheorem{cor}[thm]{\protect\corollaryname}
\theoremstyle{remark}
\newtheorem{rem}[thm]{\protect\remarkname}
\theoremstyle{plain}
\newtheorem{conjecture}[thm]{\protect\conjecturename}
\theoremstyle{definition}
\newtheorem{problem}[thm]{\protect\problemname}
\theoremstyle{plain}
\newtheorem{lem}[thm]{\protect\lemmaname}
\theoremstyle{definition}
\newtheorem{defn}[thm]{\protect\definitionname}
\theoremstyle{remark}
\newtheorem*{rem*}{\protect\remarkname}
\theoremstyle{definition}
\newtheorem{example}[thm]{\protect\examplename}
\theoremstyle{plain}
\newtheorem{prop}[thm]{\protect\propositionname}
\theoremstyle{plain}
\newtheorem{fact}[thm]{\protect\factname}

\@ifundefined{date}{}{\date{}}
\usepackage{graphicx}
\usepackage{appendix}

\usepackage{enumitem}
\usepackage[backref=page]{hyperref}
\usepackage{mathabx}

\makeatother

\providecommand{\conjecturename}{Conjecture}
\providecommand{\corollaryname}{Corollary}
\providecommand{\definitionname}{Definition}
\providecommand{\examplename}{Example}
\providecommand{\factname}{Fact}
\providecommand{\lemmaname}{Lemma}
\providecommand{\problemname}{Problem}
\providecommand{\propositionname}{Proposition}
\providecommand{\remarkname}{Remark}
\providecommand{\theoremname}{Theorem}

\begin{document}
\global\long\def\F{\mathcal{F} }%
\global\long\def\Aut{\mathrm{Aut}}%
\global\long\def\C{\mathbb{C}}%
\global\long\def\H{\mathcal{H}}%
\global\long\def\U{\mathcal{U}}%
\global\long\def\P{\mathcal{P}}%
\global\long\def\ext{\mathrm{ext}}%
\global\long\def\hull{\mathrm{hull}}%
\global\long\def\triv{\mathrm{triv}}%
\global\long\def\Hom{\mathrm{Hom}}%

\global\long\def\trace{\mathrm{tr}}%
\global\long\def\End{\mathrm{End}}%

\global\long\def\L{\mathcal{L}}%
\global\long\def\W{\mathcal{W}}%
\global\long\def\E{\mathbb{E}}%
\global\long\def\SL{\mathrm{SL}}%
\global\long\def\R{\mathbb{R}}%
\global\long\def\Z{\mathbf{Z}}%
\global\long\def\rs{\to}%
\global\long\def\A{\mathcal{A}}%
\global\long\def\a{\mathbf{a}}%
\global\long\def\rsa{\rightsquigarrow}%
\global\long\def\D{\mathbf{D}}%
\global\long\def\b{\mathbf{b}}%
\global\long\def\df{\mathrm{def}}%
\global\long\def\eqdf{\stackrel{\df}{=}}%
\global\long\def\ZZ{\mathcal{Z}}%
\global\long\def\Tr{\mathrm{Tr}}%
\global\long\def\N{\mathbb{N}}%
\global\long\def\std{\mathrm{std}}%
\global\long\def\HS{\mathrm{H.S.}}%
\global\long\def\e{\varepsilon}%
\global\long\def\d{\mathbf{d}}%
\global\long\def\AA{\mathbf{A}}%
\global\long\def\BB{\mathbf{B}}%
\global\long\def\u{\mathbf{u}}%
\global\long\def\v{\mathbf{v}}%
\global\long\def\spec{\mathrm{spec}}%
\global\long\def\Ind{\mathrm{Ind}}%
\global\long\def\half{\frac{1}{2}}%
\global\long\def\Re{\mathrm{Re}}%
\global\long\def\Im{\mathrm{Im}}%
\global\long\def\p{\mathfrak{p}}%
\global\long\def\j{\mathbf{j}}%
\global\long\def\uB{\underline{B}}%
\global\long\def\tr{\mathrm{tr}}%
\global\long\def\rank{\mathrm{rank}}%
\global\long\def\K{\mathbf{K}}%
\global\long\def\hh{\mathcal{H}}%
\global\long\def\h{\mathfrak{h}}%

\global\long\def\EE{\mathcal{E}}%
\global\long\def\PSL{\mathrm{PSL}}%
\global\long\def\G{\mathcal{G}}%
\global\long\def\Int{\mathrm{Int}}%
\global\long\def\acc{\mathrm{acc}}%
\global\long\def\awl{\mathsf{awl}}%
\global\long\def\even{\mathrm{even}}%
\global\long\def\z{\mathbf{z}}%
\global\long\def\id{\mathrm{id}}%
\global\long\def\CC{\mathcal{C}}%
\global\long\def\cusp{\mathrm{cusp}}%
\global\long\def\new{\mathrm{new}}%

\global\long\def\LL{\mathbb{L}}%
\global\long\def\M{\mathbf{M}}%
\global\long\def\I{\mathcal{I}}%
\global\long\def\X{X}%
\global\long\def\free{\mathbf{F}}%
\global\long\def\into{\hookrightarrow}%
\global\long\def\Ext{\mathrm{Ext}}%
\global\long\def\B{\mathcal{B}}%
\global\long\def\Id{\mathrm{Id}}%
\global\long\def\Q{\mathbb{Q}}%

\global\long\def\O{\mathcal{T}}%
\global\long\def\Mat{\mathrm{Mat}}%
\global\long\def\NN{\mathrm{NN}}%
\global\long\def\nn{\mathfrak{nn}}%
\global\long\def\Tr{\mathrm{Tr}}%
\global\long\def\SGRM{\mathsf{SGRM}}%
\global\long\def\m{\mathbf{m}}%
\global\long\def\n{\mathbf{n}}%
\global\long\def\k{\mathbf{k}}%
\global\long\def\GRM{\mathsf{GRM}}%
\global\long\def\vac{\mathrm{vac}}%
\global\long\def\SS{\mathcal{S}}%
\global\long\def\red{\mathrm{red}}%
\global\long\def\V{V}%

\title{Strongly convergent unitary representations of right-angled Artin
groups}
\author{Michael Magee and Joe Thomas}
\maketitle
\begin{abstract}
{\footnotesize{}We prove using a novel random matrix model that all
right-angled Artin groups have a sequence of finite dimensional unitary
representations that strongly converge to the regular representation.
We deduce that this result applies also to: the fundamental group
of a closed hyperbolic manifold that is either three dimensional or
standard arithmetic type, any Coxeter group, and any word-hyperbolic
cubulated group. }{\footnotesize\par}

{\footnotesize{}One strong consequence of these results is that any
closed hyperbolic three-manifold has a sequence of finite dimensional
flat Hermitian vector bundles with bottom of the spectrum of the Laplacian
asymptotically at least 1. }{\footnotesize\par}

{\footnotesize{}\tableofcontents{}}{\footnotesize\par}
\end{abstract}

\section{Introduction}

For $N\in\N$ let $\U(N)$ denote the group of $N\times N$ complex
unitary matrices. For a discrete group $G$, we denote by $\lambda_{G}:G\to\End(\ell^{2}(G))$
the left regular representation. We say that a sequence of unitary
representations\\
$\{\rho_{i}:G\to\U(N_{i})\}_{i=1}^{\infty}$ \emph{strongly converge
to the regular representation }if for any $z\in\mathbb{C}[G]$, 
\[
\lim_{i\to\infty}\|\rho_{i}(z)\|=\|\lambda_{G}(z)\|.
\]
The norm on the left is the operator norm on $\C^{N_{i}}$ with respect
to the standard Hermitian metric, and the norm on the right is the
operator norm on $\ell^{2}(G)$. 

Let $\Gamma$ denote a finite simple graph and $G\Gamma$ denote the
corresponding right-angled Artin group (RAAG) generated by the vertices
$V(\Gamma)$ subject to the relations that vertices commute if and
only if they are joined by an edge in $\Gamma$. 
\begin{thm}
\label{thm:strong-converge-raags}For any finite simple graph $\Gamma$,
there exists a sequence of finite dimensional unitary representations
of $G\Gamma$ that strongly converge to the regular representation.
\end{thm}

This result interpolates between $\Z^{r}$ (where it is straightforward
to prove using the Fourier transform) and finitely generated free
groups $\mathbf{F}_{r}$, where the result was obtained in a breakthrough
of Haagerup and Thorbj\o rnsen \cite[Thm. A]{HaagerupThr}. Theorem
\ref{thm:strong-converge-raags} is deduced from a random matrix result
--- Theorem \ref{thm:main-random-matrix} --- about random Hermitian
matrices in factors of tensor products who overlap in a way determined
by the graph $\Gamma$.

RAAGs are important building blocks in geometric group theory because
of their connection to CAT(0) cube complexes \cite{WiseRAAGs}. The
property of having finite dimensional unitary representations that
strongly converge to the regular representation is preserved by passing
to finite index supergroups \cite[Lemma 7.1]{louder2023strongly}
and arbitrary subgroups, so if a group $G$ virtually embeds into
a RAAG, then it also has this property. 
\begin{cor}
\label{cor:main-types}Let $G$ be one of the following types of groups.
\begin{enumerate}
\item \label{enu:The-fundamental-group-1}The fundamental group of a compact
hyperbolic three-manifold.
\item \label{enu:The-fundamental-group-2}The fundamental group of a `standard'\footnote{Here standard means that the fundamental group is a torsion-free arithmetic
lattice in $\mathbf{G}(\Q)$ where $\mathbf{G}$ is an algebraic group
over $\Q$ arising by restriction of scalars from an orthogonal group
over a totally real number field. } compact arithmetic hyperbolic manifold.
\item \label{enu:Any-Coxeter-group.}Any Coxeter group.
\item \label{enu:Any-word-hyperbolic-group}Any word-hyperbolic group acting
properly and cocompactly on a CAT(0) cube complex.
\end{enumerate}
Then there exists a sequence of finite dimensional unitary representations
of $G$ that strongly converge to the regular representation.
\end{cor}

Type \ref{enu:The-fundamental-group-1} above is contained in Type
\ref{enu:Any-word-hyperbolic-group} by works of Kahn---Markovic
\cite{KM} and Bergeron---Wise \cite[Thm. 5.3]{BW}. Type \ref{enu:The-fundamental-group-2}
is contained in Type \ref{enu:Any-word-hyperbolic-group} by Bergeron---Haglund---Wise
\cite{BHW}. Type \ref{enu:Any-Coxeter-group.} is contained in Type
\ref{enu:Any-word-hyperbolic-group} by Haglund---Wise \cite{HW2}.
The fact that groups of Type \ref{enu:Any-word-hyperbolic-group}\emph{
}virtually embed into RAAGs is a result of Agol \cite{Agol} combined
with Haglund---Wise \cite{HW}.

Theorem \ref{thm:strong-converge-raags} implies that for all RAAGs
$G$ --- as well as all those groups appearing in Corollary \ref{cor:main-types}
--- the reduced $C^{*}$-algebra $C_{\red}^{*}(G)$ is \emph{matricial
field} (MF) in the sense of Blackadar and Kirchberg \cite[Def. 3.2.1]{BK}.
Hence our results here dramatically extend the known examples of such
groups from those covered by \cite{HaagerupThr,Collins2014,Hayes,TikuisisWhiteWinter,RainoneSchafhauser,louder2023strongly,schafhauser2023finite,BordenaveCollins3}
--- see Schafhauser \cite[Introduction]{schafhauser2023finite} for
a recent survey.

Since any RAAG $G\Gamma$ contains an obvious copy of the free group
$\mathbf{F}_{2}$ on two generators as soon as $\Gamma$ is not a
complete graph, it is non-amenable in this case. Hence by \cite[\S\S 5.14]{voic_quasidiag}
(see also \cite[Rmk. 8.6]{HaagerupThr}) Theorem \ref{thm:strong-converge-raags}
shows that the \emph{$\Ext(C_{\red}^{*}(G))$ is not a group }result
from \cite{HaagerupThr} holds as soon as the RAAG $G$ is not abelian
--- in contrast to the fact that\textbf{ }$\Ext(C_{\red}^{*}(\Z^{r}))=\Ext((S^{1})^{r})$
is a group by Brown, Douglas, and Fillmore \cite[Thm. 1.23]{BrownDouglasFillmore2}. 
\begin{rem}
\label{rem:connected-component}It is easy to see that $\Hom(G\Gamma,\U(n))$
is connected as a subspace of $(\C^{n\times n})^{V}$ with the Euclidean
topology, where the subspace is induced by the images of generators
under a homomorphism in $\Mat_{n\times n}(\C)$. Since restriction
to subgroups and induction to finite index supergroups induce continuous
maps on $\Hom(H,\U(N))$ with respect to the same (Euclidean) topologies,
it follows that all representations in Corollary \ref{cor:main-types}
are in the connected component of the trivial representation in $\Hom(G,\U(N))$.
\end{rem}

Another consequence of our results is the following theorem on the
spectral geometry of hyperbolic 3-manifolds.
\begin{thm}
\label{thm:spectral-gap}Let $M=\Lambda\backslash\mathbb{H}^{3}$
be a compact hyperbolic 3-manifold. There exist a sequence of finite
dimensional unitary representations $\{\pi_{i}:\Lambda\to\U(N_{i})\}_{i=1}^{\infty}$
such that the $\pi_{i}$-twisted Laplacians $\Delta_{\pi_{i}}$ have
spectrum contained in $[1-o_{i\to\infty}(1),\infty)$.
\end{thm}

In fact, in the context of Theorem \ref{thm:spectral-gap}, random
finite dimensional unitary representations a.s. enjoy the conclusion
on the spectral gap, where `random' refers to one of the following
random models:
\begin{itemize}
\item Pick a finite index subgroup of $\Lambda'\leq\Lambda$ and an embedding
$\Lambda'\hookrightarrow G\Gamma$ for some finite simple graph $\Gamma$.
\item Construct a random unitary representation of $G\Gamma$ from the random
Hermitian matrix model of $\S\S$\ref{subsec:The-random-matrix} and
the application of the functional calculus that happens in the proof
of Theorem \ref{thm:strong-converge-raags}.
\item Pull back this random unitary representation to $\Lambda'$ and then
induce to a representation of $\Lambda$.
\end{itemize}
Although this random model is complicated and certainly not canonical,
we do believe that it should behave like a `generic unitary representation'
of $\Lambda$ modulo Remark \ref{rem:connected-component} --- this
is a heuristic and not a precise mathematical statement.

Even though Theorem \ref{thm:spectral-gap} would certainly follow
from Corollary \ref{cor:main-types} and the arguments of Hide and
the first named author (M.M.) from \cite{HM21}, with resolvent estimates
adapted to higher dimensions, that approach is complicated by the
problem of dealing with cusps, which is not present here. In $\S$\ref{sec:Proof-of-Theorem-sg}
we follow a philosophically similar but less involved method based
on representation theory to prove Theorem \ref{thm:spectral-gap}.

When Corollary \ref{cor:main-types} applies to a hyperbolic manifold
$M$, it also gives results about the $\rho_{i}$-twisted Laplace-de
Rham operators on differential $p$-forms. Since these results require
estimates on matrix coefficients of general complementary series representations
of $\mathrm{SO}(d,1)$ that are quite far away from the spirit of
this paper we do not claim them here. They will appear in a new forthcoming
work joint with Edwards and Hide.

Theorem \ref{thm:spectral-gap} is inspired by, and offers a continuation
to, the following line of results in the setting of hyperbolic 3-manifolds.
\begin{itemize}
\item There exists sequences of graphs with fixed degrees and number of
vertices tending to infinity with optimal two sided spectral gap of
the Laplacian by Lubotzky---Phillips---Sarnak \cite{LPS} and Margulis
\cite{Margulis}. These are called \emph{Ramanujan graphs.}
\item Every finite graph has a sequence of covering spaces with optimal
relative one-sided spectral gap by the resolution of weak Bilu---Linial
conjecture by Marcus---Spielman---Srivastava \cite{MSSI} --- see
also Hall---Puder---Sawin \cite{HallPuderSawin}.
\item For any finite graph, uniformly random degree $n$ covers asymptotically
almost surely have asymptotically optimal two-sided relative spectral
gaps by a result of Bordenave---Collins \cite{BordenaveCollins}
--- see also Friedman \cite{FriedmanRelative}, Puder \cite{PUDER},
and Bordenave \cite{bordenave2015new}.
\item If $M=\Gamma\backslash\mathbb{H}^{2}$ is either a conformally compact
infinite area, or non-compact finite area hyperbolic surface, then
uniformly random degree $n$ covers of $M$ asymptotically almost
surely have asymptotically optimal relative spectral gaps of the Laplacian
\cite{MageeNaud2,HM21}. In the case of infinite area, the relative
spectral gap is optimal.
\item If $M=\Gamma\backslash\mathbb{H}^{2}$ is closed, there exist a sequence
of covering spaces of $M$ with asymptotically optimal relative spectral
gaps \cite{louder2023strongly}.
\end{itemize}
To explain more the connection to these results, unitary representations
of $\Gamma$ correspond to covering spaces of $M$ whenever the representation
factors as
\begin{equation}
\Gamma\to S_{N}\xrightarrow{\mathrm{std}}\U(N-1)\label{eq:factoring}
\end{equation}
where $S_{N}$ is the group of permutations of $N$ letters, and $\mathrm{std}$
is the $N-1$ dimensional irreducible component of the representation
of $S_{N}$ by 0-1 matrices. The natural conjecture that this paper
leaves open is the following.
\begin{conjecture}
For any closed hyperbolic 3-manifold $M$, there exist a sequence
$M_{i}$ of covering spaces of $M$ with 
\[
\spec(\Delta_{M_{i}})\cap\left[0,1-o_{i\to\infty}(1)\right)=\spec(\Delta_{M})\cap\left[0,1-o_{i\to\infty}(1)\right)
\]
where the equality respects multiplicities.
\end{conjecture}

In fact the same should be true without the small $o$ terms but this
is not even known for closed hyperbolic surfaces so this conjecture
seems well out of reach at the moment.

\subsection{The random Hermitian matrix model\label{subsec:The-random-matrix}}

Let $\SGRM(n,\sigma^{2})$ denote the class of $n\times n$ complex
self-adjoint random matrices $(X_{ij})_{i,j=1}^{n}$ for which
\[
(X_{ii})_{i},\,(\sqrt{2}\Re(X_{ij}))_{i<j},\,(\sqrt{2}\Im(X_{ij}))_{i<j}
\]
are i.i.d. standard real normal random variables in $\mathcal{N}(0,\sigma^{2})$. 

Let $V=V(\Gamma)$ and let $F$ denote the set of pairs of distinct
elements of $\V$ that are not edges. For $v\in V$ let $F(v)\subset F$
denote those non-edges containing $v$. Let $m$ be a dimension parameter.
Let $(\C^{m}){}^{\otimes F}\eqdf\bigotimes_{f\in F}\C^{m}$. We think
of a factor of this tensor product as a \emph{channel} --- there
is a channel $\C^{m}$ for each pair of non-commuting vertices of
$\Gamma$. For $F'\subset F$ let $(\C^{m}){}^{\otimes F'}$ denote
the tensor factor of $(\C^{m}){}^{\otimes F}$ consisting of the $F'$
channels. 

For each $v\in V$ consider a random Hermitian matrix of the form
\[
Y_{v}\eqdf\hat{Y}_{v}\otimes\id_{(\C^{m}){}^{\otimes F\backslash F(v)}}
\]
 where $\hat{Y}_{v}\in\End((\C^{m}){}^{\otimes F(v)})$ is in $\SGRM(m^{|F(v)|},m^{-|F(v)|})$.
Here we take the liberty of writing tensor products with permuted
factors when it is clear what we mean. The following problem should
have a solution but we do not know how to obtain it\footnote{We also tried to solve the easier version of this problem, where the
channels are of the form $\C^{m(v)}$ (i.e. with different dimensions),
but did not succeed.}.
\begin{problem}
Prove that if $\Gamma$ has no vertex that is central in $G\Gamma$,
a.s. for any n.c. polynomial $p$ in $|V|$ variables
\[
\lim_{m\to\infty}\|p(Y_{v}\,:\,v\in V)\|=\|p(s_{v}:v\in V)\|_{\O_{\Gamma}},
\]
where $\O_{\Gamma}$ is the universal Toeplitz algebra associated
to $\Gamma$ defined in $\S\S$\ref{subsec:Universal-Toeplitz-algebras}
and $s_{v}$ are the semicircular elements of $\O_{\Gamma}$ defined
in $\S\S$\ref{subsec:A-subalgebra-generated}.
\end{problem}

To get traction we change our random model slightly. The idea is to
add an extra auxiliary channel for each vertex in $\V$ --- the dimensions
of these channels will go to infinity fastest. Let 
\begin{equation}
\K:V\to\N\label{eq:K-def}
\end{equation}
 be a function prescribing these new dimension parameters and let
\[
\C^{\otimes\K}\eqdf\bigotimes_{v\in V}\C^{\K(v)}.
\]
For $V'\subset V$ let $\C_{V'}^{\otimes\K}$ denote the tensor factor
of $\C^{\otimes\K}$ corresponding to the tensor product of $\C^{\K(v)}$
over elements of $V'$. For each $v\in V$ let 
\begin{align}
X_{v}^{(m,\K)} & =\tilde{X}_{v}\otimes\id_{(\C^{m}){}^{\otimes F\backslash F(v)}}\otimes\id_{\C_{V\backslash\{v\}}^{\K}}\in\End((\C^{m}){}^{\otimes F}\otimes\C^{\otimes\K}),\label{eq:rm00}\\
\tilde{X}_{v} & \in\End((\C^{m}){}^{\otimes F(v)}\otimes\C^{\K(v)})\nonumber 
\end{align}
As the reader can see, throughout the paper we will repeatedly use
canonical isomorphisms $\End(V\otimes W)\cong\End(V)\otimes\End(W)$
to describe or refer to objects. Above, we take $\{\tilde{X}_{v}:v\in V\}$
to be independent and
\[
\tilde{X}_{v}\in\SGRM\left(\K(v)m^{|F(v)|},\frac{1}{\K(v)m^{|F(v)|}}\right).
\]
\emph{This is the random model used in the rest of the paper.}

\subsection{Overview of the proof\label{subsec:Overview-of-the}}

This section is designed to give an overview of the proof of Theorem
\ref{thm:strong-converge-raags} and some commentary to explain the
inherent difficulties and the novelty of this paper.

We begin by digression to the fact that in the case that $G\Gamma$
is a product of free groups, Theorem \ref{thm:strong-converge-raags}
--- without any of its embellishments\footnote{e.g. almost sure random matrix results.}
--- can be deduced with a bit of effort from the work of Haagerup
and Thorbj\o rnsen \cite{HaagerupThr}. This relies on the basic
but important fact that one can form representations of Cartesian
products of groups by taking tensor products.

So one might naively dream that all RAAGs embed in products of free
groups and obtain Theorem \ref{thm:strong-converge-raags} in this
manner. Unfortunately, this dream is far from true\footnote{Here we specifically thank Bram Petri and Lars Louder for enlightening
conversations.} and the reason for its failure highlights one of the novelties of
this paper. Indeed, any product of free groups is \emph{commutative
transitive}, meaning that if $a$ commutes with $b$ and $b$ commutes
with $c$, then $b$ is the identity or $a$ commutes with $c$. On
the other hand, if $\Gamma$ has vertices $a$, $b$, $c$, and $d$
with edges between $a$ and $b$, $b$ and $c$, $c$ and $d$, and
no other edges between these vertices, then $[[a,c],[b,d]]$ is non-identity
in $G\Gamma$ but is killed by any homomorphism from $G\Gamma$ to
any product of free groups.

So to obtain Theorem \ref{thm:strong-converge-raags} for general
$\Gamma$ we have to come up with a random matrix model that has the
potential to be asymptotically strongly not commutative transitive,
which is a departure from previous works. This matrix model is the
one described in the previous section.

What we work with for the bulk of the paper is the random matrix model
defined in $\S\S$\ref{subsec:The-random-matrix} with a random Hermitian
matrix for each vertex of $\Gamma$. The Hermitian matrices of $v$
and $w$ commute if $v$ and $w$ are connected in $\Gamma$. The
random matrices depend on a dimension $m$ and a further set of dimension
parameters $\K$ as in (\ref{eq:K-def}). We want to prove that if
we let these dimension parameters tend to infinity in an appropriate
way we have strong convergence to some specific limit. 

We let the parameters in $\mathbf{K}$ tend to infinity much faster
than $m$ and will compare this to the result of letting $\K$ tend
to infinity while $m$ is fixed. By thinking carefully about the `blocks'
of SGRM matrices induced from splitting of the underlying vector space
as a tensor product (see $\S\S$\ref{subsec:Block-structure-of} for
this argument), we can relate the later strong limit of matrix models
--- in the strong sense of norm convergence --- to strongly convergent
limits of cartesian products of SGRM matrices. This brings us around
to basically the same thing we started thinking about (products of
free groups!), but for an entirely different reason. To get the limit
we want, with a clean a.s. random matrix result, we appeal to recent
breakthrough work of Collins---Guionnet---Parraud \cite{CollinsGuionnetParraud}.

In this limit, each $X_{v}^{(m,\K)}$ gets replaced by something ---
written in the form $L_{v}^{(m)}+L_{v}^{(m)*}$ --- that is very
roughly speaking the tensor product of some identity operator and
an $m|F(v)|\times m|F(v)|$ matrix with entries in a non-commutative
probability space \emph{all of whose upper triangular entries are
not only independent, but free from one another in the sense of free
probability theory}\footnote{To give some sense of why it might be helpful to have many free variables
(the number will tend to infinity as $m\to\infty)$, consider Voiculescu's
free central limit theorem \cite{VoiculescuSymmetries}.}\emph{. }In fact, the structure of $L_{v}^{(m)}+L_{v}^{(m)*}$ is
a little more complicated --- see equation (\ref{eq:L-def}). In
contrast, for \emph{any} $v\neq w$ the operator-valued entries of
$L_{v}^{(m)}$ and $L_{w}^{(m)}$ commute. The precise version of
this first limit is given in Theorem \ref{thm:first-convergence}.

Now let us explain the origin of the splitting $L_{v}^{(m)}+L_{v}^{(m)*}$.
When we apply Collins---Guionnet---Parraud, many independent SGRM
matrices get replaced by many independent semicircular random variables.
We realize these free semicircular random variables as sitting inside
a Cuntz-Toeplitz algebra with the vacuum state, with each semicircular
variable of form $\ell+\ell^{*}$ where $\ell$ is a creation operator.
The splitting of the limit of $X_{v}^{(m,\K)}$ as $L_{v}^{(m)}+L_{v}^{(m)*}$
is induced by the previous splitting of semicircular variables. At
this point of the proof all randomness has been removed and we move
on to the second strong limit.

The inherent difficulty of obtaining strong convergence results like
Theorem \ref{thm:strong-converge-raags} for general non-free groups
is the lack of either:
\begin{itemize}
\item Replacement of the theory of the $R$-transform that is used heavily
in e.g. \cite{HaagerupThr}, or
\item Replacement of the theory of the non-backtracking operator that is
essential in \cite{bordenave2022strong,bordenave2015new,BordenaveCollins,BordenaveCollins3}.
\end{itemize}
Here, we get around these\footnote{Philosophically, our approach is closest to $R$-transform methods,
bearing in mind that Voiculescu defines $R$-transforms in \cite{Vo1995}
essentially by recourse to Cuntz-Toeplitz algebras that are special
cases of our universal algebras.} by using a universality theorem for $C^{*}$-algebras due to Crisp
and Laca \cite{CrispLaca} (Theorem \ref{thm:universailty}). For
any free ultrafilter $\F$ on $\N$, by taking an ultraproduct along
$\F$ we extract from the sequence
\[
\{L_{v}^{(m)}:v\in V\}_{m=1}^{\infty}
\]
operators $\{\L_{v}:v\in V\}$ in an ultraproduct $C^{*}$-algebra.
We check that the universality theorem applies to the $C^{*}$-algebra
generated by the $\L_{v}$ and hence obtain our second strong limit.
For these we need to check properties \textbf{T1-T3 }of Theorem \ref{thm:universailty}
hold for the $\L_{v}$ and that they are isometries.

First we check that the $L_{v}^{(m)}$ are isometries (Lemma \ref{lem:isometry}).
\textbf{T1}, relating to commutativity of the variables,\textbf{ }is
easily seen to hold for the $L_{v}^{(m)}$ --- without taking any
limit.

The hardest property of the universality theorem to check is \textbf{T2},
relating to annihilation between certain variables,\textbf{ }and is
shown through the following statement that appears below as Proposition
\ref{prop:key}: For all non-adjacent vertices $v\not\sim w$ in $\V$,
\[
\lim_{m\to\infty}\|(L_{v}^{(m)})^{*}L_{w}^{(m)}\|=0.
\]
This is proved by an ad hoc combinatorial argument after taking powers
(amplification) to adequately bound

\[
\|[(L_{v}^{(m)})^{*}L_{w}^{(m)}(L_{w}^{(m)})^{*}L_{v}^{(m)}]^{p}\|
\]
for some fixed but sufficiently large $p\in\N$.

\textbf{T3}, a non-degeneracy condition,\textbf{ }is established by
showing in Lemma \ref{lem:T3} that each $\prod_{i=1}^{k}\left(1-L_{w_{i}}^{(m)}L_{w_{i}}^{(m)*}\right)$
has a fixed vector --- this property will pass to any ultralimit.

These results are brought together in Theorem \ref{thm:strong_convergence_technical_theorem}.
Although we forced convergence by using an ultrafilter, the fact that
the result does not depend on the ultrafilter implies that strong
convergence holds in the traditional sense.

The combination of the previous arguments yields almost sure strong
convergence of our random matrix model for some sequence of dimensional
parameters (Theorem \ref{thm:main-random-matrix}). Theorem \ref{thm:strong-converge-raags}
is deduced via a functional calculus argument shortly thereafter ---
this argument is similar to the one Haagerup and Thorbj\o rnsen use
to pass from Hermitan to unitary matrices but things are a little
subtler here. For example, \emph{(ibid.)} presupposes the existence
of a free semicircular system in a $C^{*}$-algebra with a \emph{faithful
}trace. Proposition \ref{prop:main-Cstar} supports the corresponding
argument in the current paper.

\subsection{Acknowledgments}

We warmly thank Beno\^{i}t Collins, Charles Bordenave, Alex Gamburd,
Will Hide, Lars Louder, Akihiro Miyagawa, Bram Petri, Doron Puder,
and Dimitri Shlyakhtenko for conversations, corrections, and references.

This material is based upon work supported by the National Science
Foundation under Grant No. DMS-1926686.

This project has received funding from the European Research Council
(ERC) under the European Union\textquoteright s Horizon 2020 research
and innovation programme (grant agreement No 949143).

\subsection{Notation}

For the rest of the paper $\Gamma$ is a fixed finite simple graph.
We write $V\eqdf V(\Gamma)$ and if $v\in V$, then we write $N(v)$
for the set of its neighbors. Sometimes we use the abbreviations:
f.d. (finite dimensional), n.c. (non-commutative), w.r.t. (with respect
to), a.s. (almost surely).

\section{$C^{*}$-algebraic framework}

\subsection{Universal Toeplitz algebras\label{subsec:Universal-Toeplitz-algebras}}

As explained in $\S\S$\ref{subsec:Overview-of-the}, as for Haagerup
and Thorbj\o rnsen, we work not with unitary matrices directly, but
with Hermitian matrices, and make a passage between the two by functional
calculus. This means our target limiting space is a system of semicircular
variables $s_{v}$ in a $C^{*}$-algebra $\SS_{\Gamma}$ with a faithful
trace $\tau$ and commutation relations based on the graph $\Gamma$,
i.e. $s_{v}s_{w}=s_{w}s_{v}$ for $v\sim w$ in $\Gamma$. 

Our approach to obtaining such a limit is via universality properties
of $C^{*}$-algebras. To access strong enough universality properties\footnote{Particularly, those that do not presuppose a faithful GNS representation
for some given state.} we work with a larger \emph{Toeplitz algebra} associated to $\Gamma$
that seems to have first been studied by Crisp and Laca in \cite{CrispLaca}.
Indeed, Crisp and Laca prove the following universality theorem that
we rely on in the sequel. We write $N(v)$ for the neighbors of $v$
in $\Gamma$.
\begin{thm}[{\cite[Thm. 24]{CrispLaca}}]
\label{thm:universailty}Let $\Gamma$ be a finite simple graph.
There is a unique isomorphism class of $C^{*}$-algebra $\O_{\Gamma}$
generated by elements $\{\,\ell_{v}\,:\,v\in\V\,\}$ that are isometries
($\ell_{v}^{*}\ell_{v}=1)$ and such that
\begin{description}
\item [{T1}] If $w\in N(v)$, then 
\begin{description}
\item [{$\ell_{w}\ell_{v}=\ell_{v}\ell_{w},$}]~
\item [{$\ell_{w}^{*}\ell_{v}=\ell_{v}\ell_{w}^{*}$,}]~
\end{description}
\item [{T2}] If $w\notin N(v)\cup\{v\}$, then $\ell_{w}^{*}\ell_{v}=0.$
\item [{T3}] For any $v_{1},\ldots,v_{k}\in V$
\begin{equation}
(1-\ell_{v_{1}}\ell_{v_{1}}^{*})(1-\ell_{v_{2}}\ell_{v_{2}}^{*})\cdots(1-\ell_{v_{k}}\ell_{v_{k}}^{*})\neq0.\label{eq:non-degeneracy}
\end{equation}
\end{description}
\end{thm}

Now we describe how one can concretely obtain such generators. Let
$n=|\V|$, let $\{e_{v}\,:\,v\in\V\,\}$ be the standard basis for
$\mathbb{C}^{n}$ and let $\chi_{k}^{\Gamma}(\mathbb{C}^{n})$ be
the quotient of the space $(\mathbb{C}^{n})^{\otimes k}$ by the linear
subspace $J_{k}$ spanned by the elements of the form
\begin{align}
x\otimes(e_{v}\otimes e_{w}-e_{w}\otimes e_{v})\otimes y,\label{eq:ideal-generators}
\end{align}
where $x\in(\C^{n})^{\otimes k_{1}}$, $y\in(\C^{n})^{\otimes k_{2}}$,
$k_{1},k_{2}\geq0$, $v,w\in V$, $v\in N(w)$. Here, we set $(\mathbb{C}^{n})^{\otimes0}\eqdf\mathbb{C}\Omega$
where $\Omega\neq0$ is called the \emph{vacuum vector} and use isomorphisms
\[
(\C^{n})^{\otimes0}\otimes W\cong W\cong W\otimes(\C^{n})^{\otimes0}
\]
to interpret the elements in (\ref{eq:ideal-generators}) when $k_{1}$
or $k_{2}$ is zero. By declaring the set of images of pure tensors
\[
e_{v_{1}}\otimes\cdots\otimes e_{v_{k}}
\]
in $\chi_{k}^{\Gamma}(\mathbb{C}^{n})$ to be orthonormal, we define
an inner product on each $\chi_{k}^{\Gamma}(\mathbb{C}^{n})$.

Define the \emph{configuration space for $\Gamma$, }denoted $\mathcal{H}_{\Gamma}(\mathbb{C}^{n})$,
to be the Hilbert space completion of $\bigoplus_{k\geq0}\chi_{k}^{\Gamma}(\mathbb{C}^{n})$
w.r.t. the previously defined inner product. For each $v\in\V$, the
map 
\[
x\mapsto\ell_{v}(x)\eqdf e_{v}\otimes x
\]
maps each $J_{k}$ to $J_{k+1}$ and hence descends to a \emph{creation
operator}
\[
\ell_{v}:\bigoplus_{k\geq0}\chi_{k}^{\Gamma}(\mathbb{C}^{n})\to\bigoplus_{k\geq1}\chi_{k}^{\Gamma}(\mathbb{C}^{n}).
\]
Each $\ell_{v}$ is a linear isometry w.r.t. the fixed inner product
and hence extends uniquely to a linear isometry of $\mathcal{H}_{\Gamma}(\mathbb{C}^{n})$.
The adjoint operator $\ell_{v}^{*}$ is called an \emph{annihilation
operator}. The action of $\ell_{v}^{*}$ is --- up to the $J_{k}$
--- to remove $e_{v}$ from the front of a pure tensor if it is present
or can be commuted to the leftmost position in the tensor modulo the
$J_{k}$, and mapping the pure tensor to $0$ otherwise. It is straightforward
to check that the $\ell_{v}$ satisfy \textbf{T1 }and \textbf{T2,
}and\textbf{ }the vacuum vector is fixed by the left hand side of
(\ref{eq:non-degeneracy}), so \textbf{T3 }holds. 

Therefore, we may as well think of the universal Toeplitz algebra
$\O_{\Gamma}$ in this concrete form. In particular, we have a state
\[
\tau_{\vac}(t)\eqdf\langle t\Omega,\Omega\rangle
\]
on $\O_{\Gamma}$ that has a faithful GNS representation, namely,
the one we just described on $\H_{\Gamma}$.
\begin{rem}
\label{rem:CT-algebra}When $\Gamma$ is a graph with no edges on
$n$ vertices, the resulting $\O_{n}\cong\O_{\Gamma}$ is the universal
\emph{Cuntz-Toeplitz} $C^{*}$-algebra generated by $n$ isometries
with mutually orthogonal ranges, introduced in \cite[\S 3]{Cuntz}.
In this case the representation $\H_{n}\eqdf\H_{\Gamma}(\C^{n})$
is the Fock space of Boltzmann statistics.
\end{rem}

\begin{lem}
\label{lem:normal-form-A}For any $W\subset\V$, any noncommutative
monomial in the $\ell_{v}$ and $\ell_{v}^{*}$ with $v\in W$ is
either equal to zero or equal to some
\[
\ell_{v_{1}}\cdots\ell_{v_{p}}\ell_{w_{1}}^{*}\cdots\ell_{w_{q}}^{*}
\]
where all $v_{i}$ and $w_{i}$ are in $W$. We refer to such a form
for the monomial as \textbf{normal form}.
\end{lem}

\begin{rem}
This normal form is unique up to permutations of the form $\ell_{v_{k}}\ell_{v_{k+1}}=\ell_{v_{k+1}}\ell_{v_{k}}$
and $\ell_{v_{k}}^{*}\ell_{v_{k+1}}^{*}=\ell_{v_{k+1}}^{*}\ell_{v_{k}}^{*}$
whenever $v_{k}\in N(v_{k+1})$. However, we do not need this fact
in the sequel.
\end{rem}

\begin{proof}[Proof of Lemma \ref{lem:normal-form-A}]
Either the monomial is already in normal form or there is an occurrence
of $\ell_{v}^{*}\ell_{w}$ for some vertices $v$ and $w$. By \textbf{T2},
this occurrence either makes the monomial zero, or it can be reduced
to the identity, or the two elements can be commuted. In any case,
the occurrence can be removed. Iterating this gives the result.
\end{proof}
It is worth noting that 
\[
\tau_{\vac}\left(\ell_{v_{1}}\cdots\ell_{v_{p}}\ell_{w_{1}}^{*}\cdots\ell_{w_{q}}^{*}\right)=0
\]
if and only if $p+q>0$, and this is another characterization of $\tau_{\vac}$.

\subsection{Right-angled probability}

The development of our framework relies on the following definitions. 
\begin{defn}
We say that a sequence $(v_{j})_{j=1}^{m}\subseteq\V$ is $\Gamma$\emph{-reduced}
if whenever $v_{j}=v_{k}$ for $j<k$, then there exists $j<l<k$
such that $v_{l}\notin N(v_{j})=N(v_{k})$.
\end{defn}

\begin{defn}
\label{def:right-angled}Let $\mathcal{A}$ be a unital $C^{*}$-algebra
and $\tau$ a state on $\mathcal{A}$. Given a finite simple graph
$\Gamma$ on the vertices $\{1,\ldots,n\}$, we say that a collection
of unital $*$-subalgebras $\{\mathcal{A}_{v}\}_{v\in\V}$ of $\mathcal{A}$
is \emph{$\Gamma$-right-angled w.r.t. $\tau$ }if 
\begin{enumerate}
\item $\mathcal{A}_{v}$ commutes with $\mathcal{A}_{w}$ whenever $w\in N(v)$,
\item \label{enu:For-any-}For any $m\in\mathbb{N}$ if a sequence $(v_{j})_{j=1}^{m}\subseteq\V$
is $\Gamma$-reduced then for any $a_{j}\in\mathcal{A}_{v_{j}}$ with
$\tau(a_{j})=0$, one has $\tau(a_{1}\cdots a_{n})=0.$
\end{enumerate}
\end{defn}

\begin{rem*}
The case of graphs with no edges in Definition \ref{def:right-angled}
recovers the definition of freeness. Definition \ref{def:right-angled}
is a rewording of the definition of $\epsilon$-independence given
in \cite[Defn. 3.2]{SW}, see M{\l}otkowski \cite[Defn. 3]{Mlot}
for the origin of this definition. 
\end{rem*}
\begin{example}
\label{exa:The--subalgebras-of}\cite[Prop. 4.2]{SW} The $*$-subalgebras
of $C_{\red}^{*}(G\Gamma)$ generated by the individual elements of
$\V$ are $\Gamma$-right-angled w.r.t. the canonical tracial state. 
\end{example}

\begin{lem}
\label{lem:The--subalgebras-right-angled}The $C^{*}$-subalgebras
$\A_{v}\eqdf\overline{\langle\ell_{v},\ell_{v}^{*}\rangle}$ of $\O_{\Gamma}$
are $\Gamma$-right-angled w.r.t. the state $\tau_{\vac}$.
\end{lem}

\begin{proof}
It is clear that if $v\in N(w)$ then $\mathcal{A}_{v}$ and $\mathcal{A}_{w}$
commute with one another by \textbf{T1}. 

Suppose that a sequence $(v_{j})_{j=1}^{m}\subseteq\V$ is $\Gamma$-reduced,
and $a_{j}\in\mathcal{A}_{v_{j}}$ with $\tau_{\vac}(a_{j})=0$. We
want to prove $\tau_{\vac}(a_{1}\cdots a_{n})=0.$ Since for each
$v\in\V$, normal form monomials involving only $\ell_{v}$ and $\ell_{v}^{*}$
span a dense subspace of $\A_{v}$ by Lemma \ref{lem:normal-form-A},
by continuity and linearity of $\tau_{\vac}$ we can reduce to checking
the case where $a_{j}$ are normal form monomials in the generators
of the respective subalgebras $\mathcal{A}_{v_{j}}$, i.e., $a_{j}=\ell_{v_{j}}^{c_{j}}\ell_{v_{j}}^{*d_{j}}$
for some $c_{j}+d_{j}>0$. We then are required to show that
\begin{equation}
\tau_{\vac}(\ell_{v_{1}}^{c_{1}}\ell_{v_{1}}^{*d_{1}}\cdots\ell_{v_{m}}^{c_{m}}\ell_{v_{m}}^{*d_{m}})=0.\label{eq:supposed-to-be-zero}
\end{equation}

First suppose some $d_{k}>0$ and let $k$ be the maximal element
of $\{1,\ldots,m\}$ for which this is true. If there is no $k<l\leq m$
such that $v_{l}\notin N(v_{k})$ then $\ell_{v_{k}}^{*d_{k}}\cdots\ell_{v_{m}}^{c_{m}}\Omega=\cdots\ell_{v_{m}}^{c_{m}}\ell_{v_{k}}^{*d_{k}}\Omega=0$.
Otherwise pick the minimum $l>k$ with $v_{l}\notin N(v_{k})$. Since
the sequence is $\Gamma$-reduced, $v_{l}\neq v_{k}$ (or else there
would be a smaller $l$). Then $c_{l}>0$ (as $d_{\ell}=0$) and $\ell_{v_{k}}^{*d_{k}}\cdots\ell_{v_{\ell}}^{c_{\ell}}=\cdots\ell_{v_{k}}^{*d_{k}}\ell_{v_{l}}^{c_{l}}=0$
by \textbf{T2}. Hence if any $d_{k}>0$ then (\ref{eq:supposed-to-be-zero})
holds. Otherwise, all $d_{k}=0$. But then all $c_{k}>0$ and
\[
\tau_{\vac}(\ell_{v_{1}}^{c_{1}}\ell_{v_{1}}^{*d_{1}}\cdots\ell_{v_{m}}^{c_{m}}\ell_{v_{m}}^{*d_{m}})=\langle\ell_{v_{1}}^{c_{1}}\cdots\ell_{v_{m}}^{c_{m}}\Omega,\Omega\rangle=\langle\ell_{v_{m}}^{*c_{m}}\cdots\ell_{v_{1}}^{*c_{1}}\Omega,\Omega\rangle=0.
\]
\end{proof}
Like in the case of free subalgebras, the restrictions of a state
to $\Gamma$-right-angled subalgebras that generate the entire $C^{*}$-algebra
are enough to determine the state on the entire $C^{*}$-algebra.
The proof of this fact is similar to \cite[Proposition 1.3]{Vo1995},
and an outline of the proof is given in the paragraphs proceeding
Definition 3.2 in \cite{SW}, we include the details for completeness
here.
\begin{prop}
\label{prop:right-angle-determine-state}Suppose that $\mathcal{A}$
is a unital $C^{*}$-algebra, $\tau$ a state on $\mathcal{A}$, $\Gamma$
a finite simple graph on the vertex set $\{1,\ldots,n\}$ and $\{\mathcal{A}_{v}\}_{v\in V}$
a $\Gamma$-right-angled collection of $*$-subalgebras w.r.t. $\tau$
such that $\mathcal{A}$ is generated as a $C^{*}$-algebra by $\bigcup_{v\in V}\mathcal{A}_{v}$.
Then the state $\tau$ is determined by its restrictions $\tau|_{\mathcal{A}_{v}}$.
\end{prop}

\begin{proof}
By hypothesis, any $a\in\mathcal{A}$ can be written as the limit
of finite linear combinations of finite products of elements from
the $*$-subalgebras $\mathcal{A}_{v}$. By continuity and linearity
of $\tau$ it thus suffices to show that $\tau(a_{1}\cdots a_{m})$
can be determined by the restrictions for any selection of $a_{j}\in\mathcal{A}_{v_{j}}$
for $1\leq j\leq m$. We proceed by induction on the smallest non-negative
integer $k$ such that $\tau(a_{j})=\tau|_{\mathcal{A}_{v_{j}}}(a_{j})=0$
for every $j>k$ and such that $(v_{j})_{j=k+1}^{m}$ is a $\Gamma$-reduced
sequence.

The base case of $k=0$ follows immediately by definition of the subalgebras
being $\Gamma$-right-angled w.r.t. $\tau$ since then $\tau(a_{j})=0$
for all $j=1,\ldots,m$ and $(v_{j})_{j=1}^{m}$ is $\Gamma$-reduced
resulting in $\tau(a_{1}\cdots a_{m})=0$.

For the inductive step, assume that the result holds up to some $k$
so that for $k+1$ we have two possibilities.

\textbf{Case 1. }$(v_{j})_{j=k+1}^{m}$ is $\Gamma$-reduced. Then
we may write
\begin{align*}
\tau(a_{1}\cdots a_{m}) & =\tau(a_{1}\cdots a_{k}\tau|_{\mathcal{A}_{v_{k+1}}}(a_{k+1})a_{k+2}\cdots a_{m})\\
 & +\tau(a_{1}\cdots a_{k}(a_{k+1}-\tau|_{\mathcal{A}_{v_{k+1}}}(a_{k+1}))a_{k+2}\cdots a_{m}).
\end{align*}
By linearity of $\tau$, the first term is equal to $\tau|_{\mathcal{A}_{v_{k+1}}}(a_{k+1})\tau(a_{1}\cdots a_{k}a_{k+2}\cdots a_{m})$
and by assumption, $\tau(a_{k+2})=\ldots=\tau(a_{m})=0$ and the sequence
$(v_{j})_{j=k+2}^{m}$ is $\Gamma$-reduced. Thus this first term
is by the inductive hypothesis completely determined by the restricted
states. For the second term, we have $\tau(a_{k+1}-\tau|_{\mathcal{A}_{v_{k+1}}}(a_{k+1}))=\tau(a_{k+2})=\ldots=\tau(a_{m})=0$
and the sequence $(v_{j})_{j=k+1}^{m}$ is $\Gamma$-reduced so that
by the inductive hypothesis the second term is also completely determined
by the restricted states and hence $\tau(a_{1}\cdots a_{m})$ is also.

\textbf{Case 2. }$(v_{j})_{j=k+1}^{m}$ is not $\Gamma$-reduced but
by the inductive hypothesis, $(v_{j})_{j=k+2}^{m}$ is $\Gamma$-reduced.
Then, it must be the case that there exists some $k+1<l\leq m$ for
which $v_{l}=v_{k+1}$ and $v_{c}\in N(v_{k+1})=N(v_{l})$ for every
$k+1<c<l$. But by definition of $\Gamma$-right-angled, $a_{k+2},\ldots,a_{l-1}$
all commute with $a_{l}$ and so 
\[
a_{1}\cdots a_{m}=a_{1}\cdots a_{k}a_{k+1}a_{l}a_{k+2}\cdots a_{l-1}a_{l+1}\cdots a_{m}.
\]
We can then decompose the state of $a_{1}\cdots a_{m}$ in the following
manner
\begin{align*}
\tau(a_{1}\cdots a_{m}) & =\tau(a_{1}\cdots a_{k}(a_{k+1}a_{l}-\tau|_{\mathcal{A}_{v_{k+1}}}(a_{k+1}a_{l}))a_{k+2}\cdots a_{l-1}a_{l+1}\cdots a_{m})\\
 & +\tau|_{\mathcal{A}_{v_{k+1}}}(a_{k+1}a_{l})\tau(a_{1}\cdots a_{k}a_{k+2}\cdots a_{l-1}a_{l+1}\cdots a_{m}).
\end{align*}
Then, the sequence $(v_{k+1},\ldots,v_{l-1},v_{l+1},\ldots,v_{m})$
is $\Gamma$-reduced. To prove this, suppose first that there exists
some $b\in\{k+2,\ldots,l-1,l+1,\ldots,m\}$ such that $v_{k+1}=v_{b}$
and $v_{c}\in N(v_{k+1})=N(v_{b})$ for all $k+1<c<b$, $c\neq l$.
Then $b>l$ since by construction, if \textbf{$k+1<b<l$} , then $v_{b}=v_{k+1}=v_{\ell}$
and $v_{b}\in N(v_{l})$ contradicting $\Gamma$ being simple. But
when $b>l$, since $v_{k+1}=v_{l}$, we obtain $v_{l}=v{}_{b}$ and
this gives a contradiction to $\{v_{j}\}_{j=k+2}^{m}$ being $\Gamma$-reduced
as for each $l+1\leq c<b$ we have $v_{c}\in N(v_{b})$. The only
other way that $(v_{k+1},\ldots,v_{l-1},v_{l+1},\ldots,v_{m})$ could
not be $\Gamma$-reduced without contradicting the fact that $\{v_{j}\}_{j=k+2}^{m}$
is $\Gamma$-reduced, is if there exists $a<l<b$ such that $v_{a}=v_{b}$
and $v_{l}$ is the only index for which $v_{l}\notin N(v_{a})=N(v_{b})$,
but this is not true since by construction, $v_{l}\in N(v_{a})$.
Thus, the sequence is indeed $\Gamma$-reduced. \\
We can then conclude that $\tau(a_{1}\cdots a_{m})$ is determined
by the restricted states because for the first term we have, $a_{k+1}a_{l}-\tau|_{\mathcal{A}_{v_{k+1}}}(a_{k+1}a_{l})\in\mathcal{A}_{v_{k+1}}$,
$\tau(a_{k+1}a_{l}-\tau|_{\mathcal{A}_{v_{k+1}}}(a_{k+1}a_{l}))=0=\tau(a_{k+2})=\ldots=\tau(a_{l-1})=\tau(a_{l+1})=\ldots=\tau(a_{m})$
and $(v_{k+1},\ldots,v_{l-1},v_{l+1},\ldots,a_{m})$ is $\Gamma$-reduced
which means by the inductive hypothesis, 
\[
\tau(a_{1}\cdots a_{k}(a_{k+1}a_{l}-\tau|_{\mathcal{A}_{v_{k+1}}}(a_{k+1}a_{l}))a_{k+2}\cdots a_{l-1}a_{l+1}\cdots a_{m})
\]
 is determined by the restricted states and similarly for the second
term. 
\end{proof}

\subsection{A subalgebra generated by semicircular variables\label{subsec:A-subalgebra-generated}}

Let
\[
s_{v}\eqdf\ell_{v}+\ell_{v}^{*}\in\O_{\Gamma}.
\]
We call these \emph{semicircular variables} in light of Lemma \ref{lem:semicircular-right-angled-system}
below. Let $\SS_{\Gamma}$ denote the unital $C^{*}$-algebra generated
by the $s_{v}$ (and $1$) in $\O_{\Gamma}$. 
\begin{lem}
\label{lem:semicircular-right-angled-system}W.r.t. the state $\tau_{\vac}$
restricted to $\SS_{\Gamma}$, each $s_{v}$ is distributed according
to the semicircle law, that is, for each even $k\in\N$
\begin{equation}
\tau_{\vac}(s_{v}^{k})=\frac{1}{2\pi}\int_{-2}^{2}t^{k}\sqrt{4-t^{2}}\mathrm{d}t\label{eq:semicircle-law}
\end{equation}
 and for odd $k$, the moment is zero.
\end{lem}

\begin{proof}
It follows by evaluating $\langle\left(\ell_{v}+\ell_{v}^{*}\right)^{2p}\Omega,\Omega\rangle$
as the Catalan number $C_{p}$ via Dyck words and then using the formula
$C_{p}=\frac{1}{2\pi}\int_{-2}^{2}t^{2p}\sqrt{4-t^{2}}\mathrm{d}t$
. (This is well-known in the free case, which uses the same argument.) 
\end{proof}
\begin{lem}
\label{lem:The-GNS-representation}The GNS representation of $\SS_{\Gamma}$
w.r.t. the restriction of the state $\tau_{\vac}$ is faithful.
\end{lem}

\begin{proof}
It suffices to prove that $\SS_{\Gamma}.\Omega$ is dense in $\H_{\Gamma}$.
The set of all
\[
e_{v_{1}}\otimes\cdots\otimes e_{v_{k}}
\]
where $k\in\N\cup\{0\}$ span a dense subspace of $\H_{\Gamma}$,
so it suffices to prove each of these is in $\SS_{\Gamma}.\Omega$.
We prove this by induction on $k$. When $k=0$, the statement is
that the vacuum vector is in $\SS_{\Gamma}.\Omega$, this holds since
$\SS_{\Gamma}$ is unital.

So assume $k>0$ and that for $0\leq K\leq k-1$ we have $I_{K}\subset\SS_{\Gamma}.\Omega$
where $I_{K}$ denotes the subspace spanned by all
\[
e_{w_{1}}\otimes\cdots\otimes e_{w_{K}}.
\]
Let $\{v_{i}\}_{i=1}^{k}$ be a sequence of vertices. Then
\begin{align*}
s_{v_{1}}s_{v_{2}}\cdots s_{v_{k}}\Omega & =\text{\ensuremath{\left(\ell_{v_{1}}+\ell_{v_{1}}^{*}\right)}\ensuremath{\ensuremath{\cdots\left(\ell_{v_{k}}+\ell_{v_{k}}^{*}\right)\Omega}}}\\
 & \in\ell_{v_{1}}\ell_{v_{2}}\cdots\ell_{v_{k}}\Omega+\sum_{i=0}^{k-1}I_{i}=e_{v_{1}}\otimes\cdots\otimes e_{v_{k}}+\sum_{i=0}^{k-1}I_{i}.
\end{align*}
Hence rearranging we obtain $e_{v_{1}}\otimes\cdots\otimes e_{v_{k}}\in\SS_{\Gamma}.\Omega$.
\end{proof}
Let 
\[
\varphi(t)\eqdf\begin{cases}
-\pi & t\leq-2\\
\int_{0}^{t}\sqrt{4-s^{2}}ds & -2\leq t\leq2\\
\pi & 2\leq t
\end{cases}
\]
 and 
\[
\psi(t)\eqdf\exp(i\varphi(t)),
\]
so $\psi:\R\to S^{1}$ is $C^{1}$. The map $\psi$ is a bijection
when restricted to $(-2,2]$. Let $\psi^{-1}:S^{1}\to(-2,2]$ denote
the corresponding inverse. Each $v\in\V$ corresponds to a unitary
$\lambda(v)\in\U(\ell^{2}(G\Gamma))$ where $\lambda$ is the left
regular representation. We are now ready to prove the main result
of this $\S\S$.
\begin{prop}
\label{prop:main-Cstar}The assignment $\lambda(v)\mapsto\psi(s_{v})$
extends to an injective $*$-homomorphism
\[
C_{\red}^{*}(G\Gamma)\hookrightarrow\SS_{\Gamma}.
\]
Moreover, the state $\tau_{\vac}$ on $\SS_{\Gamma}$ is a faithful
trace. 
\end{prop}

\begin{proof}
The group von Neumann algebra\emph{ $L(G\Gamma)$ }in $\B(\ell^{2}(G\Gamma))$
has faithful trace $\tau_{G\Gamma}(a)\eqdf\langle a\delta_{e},\delta_{e}\rangle$.
For $v\in V$ let $S_{v}\eqdf\psi^{-1}(\lambda(v))$ be the result
of applying Borel functional calculus for normal operators to $\psi^{-1}$
and $\lambda(v)$ --- note that $S_{v}$ is in the von Neumann algebra
$L(\lambda(v))$ generated by $\lambda(v)$ in $L(G\Gamma)$. Let
$\SS'_{\Gamma}$ denote the $C^{*}$-algebra generated by the $S_{v}$. 

\emph{Claim 1. }Each $S_{v}$ is semicircular w.r.t. $\tau_{G\Gamma}$.
This is because, as in \cite[\S 8]{HaagerupThr} the push-forward
$\psi_{*}^{-1}(\mathrm{Haar}_{S^{1}})$ has density $\varphi'(t)\mathrm{d}t$
on the real line, hence is the semicircle law. On the other hand,
by the Borel functional calculus, $\psi_{*}^{-1}(\mathrm{Haar}_{S^{1}})$
is the law of $S_{v}$. 

\emph{Claim 2. }The GNS representation of $\SS'_{\Gamma}$ w.r.t.
$\tau_{G\Gamma}$ is faithful --- indeed the state $\tau_{G\Gamma}$
is faithful on $\SS'_{\Gamma}$ being the restriction of a faithful
state.

\emph{Claim 3. }The algebras $\A_{v}$ generated by the $S_{v}$ are
$\Gamma$-right-angled subalgebras w.r.t. $\tau_{G\Gamma}$. 

\emph{Proof of Claim 3. }To see the first condition of Definition
\ref{def:right-angled}, if $v\sim w$ in $\Gamma$, it is a basic
fact (following e.g. from the double commutant theorem) that $\lambda(v),\lambda(w)$
commuting imply that $L(\lambda(v))$ and $L(\lambda(w))$ are commuting
von Neumann algebras. These contain $\A_{v}$ and $\A_{w}$ respectively. 

For the second condition of Definition \ref{def:right-angled}, consider
$q\geq0$ such that there exist $\Gamma$-reduced $(v_{j})_{j=1}^{m}\subseteq\V$,
$a_{j}\in\mathcal{A}_{v_{j}}$ with $\tau_{G\Gamma}(a_{j})=0$, $a_{i}$
is a $*$-polynomial of $\lambda(v_{i})$ for $i\leq m-q$, and 
\[
\tau_{G\Gamma}(a_{1}\cdots a_{m})\neq0.
\]
By Example \ref{exa:The--subalgebras-of}, the set of such $q$ does
not contain zero. The second condition of Definition \ref{def:right-angled}
is equivalent to there being no such $q$, so for the sake of a contradiction
suppose $q>0$ is minimal such that the above holds, and let other
notation be as before. Let $k\eqdf m-q$. The mapping 
\[
b\mapsto\tau_{G\Gamma}(a_{1}\cdots a_{k}ba_{k+2}\cdots a_{m})
\]
is continuous in $b$ since $\tau_{G\Gamma}$ is obviously normal.
Since $a_{k+1}$ is a polynomial in $S_{v_{k+1}}$, it is in the von
Neumann algebra $L(\lambda(v_{k+1}))$ and hence can be approximated
in the weak operator topology by $b$ that is a $*$-polynomial of
$\lambda(v_{k+1})$. Furthermore, since $\tau_{G\Gamma}(a_{k+1})=0$,
by replacing $b$ by $b-\tau_{G\Gamma}(b)$ in this approximation
we can assume $\tau_{G\Gamma}(b)=0$. Then 
\[
0\neq\tau_{G\Gamma}(a_{1}\cdots a_{m})=\lim_{b\stackrel{\mathrm{W.O.T.}}{\to}a_{k+1}}\tau_{G\Gamma}(a_{1}\cdots a_{k}ba_{k+2}\cdots a_{m})=0
\]
is a contradiction. \emph{This ends the proof of Claim 3.}

Claim 1 and Claim 3 above imply that the joint distribution of the
$S_{v}$ w.r.t. $\tau_{G\Gamma}$ is the same as that of the $s_{v}$
w.r.t. $\tau_{\vac}$ by Proposition \ref{prop:right-angle-determine-state}
and Lemmas \ref{lem:The--subalgebras-right-angled} and \ref{lem:semicircular-right-angled-system}.
Since both $(\SS_{\Gamma},\tau_{\vac})$ --- by Lemma \ref{lem:The-GNS-representation}
--- and $(\SS'_{\Gamma},\tau_{G\Gamma})$ --- by Claim 2 --- have
faithful GNS representations, by the remark made by Voiculescu in
\cite[Rmk. 1.8]{VoiculescuCircular} the mapping $S_{v}\mapsto s_{v}$
extends to a state-preserving isomorphism
\[
(\SS'_{G\Gamma},\tau_{G\Gamma})\cong(\SS_{\Gamma},\tau_{\vac})
\]
 of $C^{*}$-algebras. This shows that $\tau_{\vac}$ is a faithful
trace on $\SS_{\Gamma}$. The isomorphism above sends
\[
\psi(S_{v})=\lambda(v)\mapsto\psi(s_{v}).
\]
Note that $\psi$ is continuous here so $\psi(s_{v})\in\SS_{\Gamma}$
by the continuous functional calculus.
\end{proof}

\section{First strong limit}

\subsection{Random matrices}

We follow \cite{HaagerupThr} and say that all random matrices are
real or complex matrix valued random variables on the same probability
space $(\Omega,\F,P)$. We say an event holds almost surely (a.s.)
if there is a $P$-null set $N\subset\Omega$ such that it holds outside
$N$. Let $\GRM(n,\sigma^{2})$ be the class of $n\times n$ complex
random matrices whose entries are i.i.d. complex normal random variables
in $\mathcal{CN}(0,\sigma^{2})$ and recall the definition of $\SGRM(n,\sigma^{2})$
from $\S\S$\ref{subsec:The-random-matrix}. Two basic facts will
be used:
\begin{fact}
\label{fact:GRMtoSGRM}If $Y$ is in $\GRM(n,\sigma^{2})$ then 
\[
X_{1}=\frac{1}{\sqrt{2}}(Y+Y^{*}),\,X_{2}=\frac{-i}{\sqrt{2}}\left(Y-Y^{*}\right)
\]
are independent elements of $\SGRM(n,\sigma^{2})$, and 
\[
Y=\frac{1}{\sqrt{2}}\left(X_{1}+iX_{2}\right).
\]
\end{fact}

\begin{fact}
\label{fact:SGRMtoGRM}If $X_{1},X_{2}$ are independent elements
of $\SGRM(n,\sigma^{2})$ then $Y=\frac{1}{\sqrt{2}}(X_{1}+iX_{2})$
is in $\GRM(n,\sigma^{2})$ and $\frac{1}{\sqrt{2}}(Y+Y^{*})=X_{1}$.
In particular, if $X\in\SGRM(n,\sigma^{2})$ there is $Y\in\GRM(n,\sigma^{2})$
such that 
\[
X=\frac{1}{\sqrt{2}}\left(Y+Y^{*}\right).
\]
\end{fact}

\subsection{Block structure of SGRM matrices\label{subsec:Block-structure-of}}

Recall from $\S\S\ref{subsec:The-random-matrix}$ that $\tilde{X}_{v}\in\SGRM\left(\K(v)m^{|F(v)|},\frac{1}{\K(v)m^{|F(v)|}}\right).$
We now want to understand what random matrices $B_{\e}$ we get if
we decompose
\[
\tilde{X}_{v}=\sum_{\e}\varepsilon\otimes B_{\e}\in\End((\C^{m})^{\otimes F})
\]
where $\e$ run over matrix units in $\End((\C^{m}){}^{\otimes F(v)})$.
One nice way to do this is as follows. We will use the standard basis
of $(\C^{m}){}^{\otimes F}$ coming from the simple tensors of standard
bases of $\C^{m}$. If $F'\subset F$ and $I,J\in[m]^{F'}$ let 
\[
e_{I}\eqdf\otimes_{f\in F'}e_{I(f)}\in(\C^{m}){}^{\otimes F'}
\]
where $e_{i}$ are the standard orthonormal basis of $\C^{m}$ w.r.t.
the fixed standard Hermitian form. We write $\check{e}_{I}$ for the
dual vector to $e_{I}$ and 
\[
\e_{IJ}\eqdf e_{I}\otimes\check{e}_{J}\in\End((\C^{m}){}^{\otimes F'}).
\]

By Fact \ref{fact:SGRMtoGRM} we can write
\begin{equation}
\tilde{X}_{v}=\frac{1}{\sqrt{2}}(R_{v}+R_{v}^{*})\label{eq:rm0}
\end{equation}
where $(R_{v})_{v\in\V}$ are independent elements of $\GRM\left(\K(v)m^{|F(v)|},\frac{1}{\K(v)m^{|F(v)|}}\right)$.
Now the problem is easier because each $R_{v}$ has no symmetry. We
have 
\begin{equation}
R_{v}=\frac{1}{\sqrt{m^{|F(v)|}}}\sum_{\substack{I,J\in[m]^{F(v)}}
}\e_{IJ}\otimes Q_{IJ}^{v}\label{eq:rm1}
\end{equation}
where $(Q_{IJ}^{v})_{I,J\in[m]^{F(v)},v\in\V}$ are independent and
$Q_{IJ}^{v}\in\GRM\left(\K(v),\frac{1}{\K(v)}\right)$. Now, to return
to $\SGRM$ matrices, by Fact \ref{fact:GRMtoSGRM} we have
\begin{equation}
Q_{IJ}^{v}=\frac{1}{\sqrt{2}}\left(X_{IJ}^{v}+iY_{IJ}^{v}\right)\label{eq:rm2}
\end{equation}
where $(X_{IJ}^{v})_{I,J\in[m]^{F(v)},v\in\V},(Y_{IJ}^{v})_{I,J\in[m]^{F(v)},v\in\V}$
are independent and 
\[
X_{IJ}^{v},Y_{IJ}^{v}\in\SGRM\left(\K(v),\frac{1}{\K(v)}\right).
\]
Combining (\ref{eq:rm00}), (\ref{eq:rm0}), (\ref{eq:rm1}), and
(\ref{eq:rm2}) we obtain

\begin{align}
X_{v}^{(m,\K)} & =\frac{1}{2\sqrt{m^{|F(v)|}}}\id_{(\C^{m}){}^{\otimes F\backslash F(v)}}\otimes\sum_{\substack{I,J\in[m]^{F(v)}}
}\left(\e_{IJ}+\e_{JI}\right)\otimes[X_{IJ}^{v}]\otimes\id_{\C_{V\backslash\{v\}}^{\otimes\K}}\nonumber \\
 & \,\,\,\,\,\,\,\,\,\,\,\,\,\,\,\,\,\,\,\,\,\,\,\,\,\,\,\,\,\,\,\,\,\,\,\,\,\,\,\,+i\left(\e_{IJ}-\e_{JI}\right)\otimes[Y_{IJ}^{v}]\otimes\id_{\C_{V\backslash\{v\}}^{\otimes\K}}.\label{eq:X-def}
\end{align}

\subsection{First strong convergence result}

For each $v\in V$ let $\O_{v}(m)\eqdf\O_{2m^{2|F(v)|}}$ denote the
Cuntz-Toeplitz $C^{*}$-algebra --- as in Remark \ref{rem:CT-algebra}
--- generated by $2m^{2|F(v)|}$ free creation operators 
\begin{align*}
\{x_{IJ}^{v+},\,x_{IJ}^{v-} & :\,I,J\in[m]^{F(v)}\,\}.
\end{align*}
Let
\[
\O^{V}(m)\eqdf\bigotimes_{v}^{\min}\O_{v}(m).
\]
\emph{Every tensor product of $C^{*}$-algebras in this paper, including
the above, is the minimal (spatial) tensor product.}

Motivated by (\ref{eq:X-def}), let
\begin{align*}
r_{IJ}^{+} & \eqdf\e_{IJ}+\e_{JI},\\
r_{IJ}^{-} & \eqdf i(\e_{IJ}-\e_{JI}).
\end{align*}
Still with (\ref{eq:X-def}) in mind, we consider elements $L_{v}^{(m)}\in\End((\C^{m}){}^{\otimes F})\otimes\O^{V}(m)$
\begin{align}
L_{v}^{(m)} & \eqdf\frac{\id_{(\C^{m}){}^{\otimes F\backslash F(v)}}}{2\sqrt{m^{|F(v)|}}}\otimes\left(\sum_{\substack{I,J\in[m]^{F(v)}}
}r_{IJ}^{+}\otimes x_{IJ}^{v+}+r_{IJ}^{-}\otimes x_{IJ}^{v-}\right)\otimes\id_{\bigotimes_{w\neq v}\O_{w}(m)},\label{eq:L-def}\\
L_{v}^{(m)*} & =\frac{\id_{(\C^{m}){}^{\otimes F\backslash F(v)}}}{2\sqrt{m^{|F(v)|}}}\otimes\left(\sum_{\substack{I,J\in[m]^{F(v)}}
}r_{IJ}^{+}\otimes[x_{IJ}^{v+}]^{*}+r_{IJ}^{-}\otimes[x_{IJ}^{v-}]^{*}\right)\otimes\id_{\bigotimes_{w\neq v}\O_{w}(m)}.\nonumber 
\end{align}

We need the following result of Collins, Guionnet, and Parraud \cite[Thm. 1.2, pt. 2]{CollinsGuionnetParraud}
that we recall for the convenience of the reader. Let $\tr$ denote
normalized trace on matrices.
\begin{thm}[Collins---Guionnet---Parraud]
\label{thm:CGP}Suppose $X_{1}^{(N)},\ldots,X_{d}^{(N)}$ are independent
elements of $\SGRM(N,N^{-1})$. Let $(s_{1},\ldots,s_{d})$ be a tuple
of free semcircular random variables in a $C^{*}$-probability space
$(\SS,\tau)$. Suppose that $Y_{1}^{(M)},\ldots,Y_{D}^{(M)}$ are
random matrices of dimension $M$ on $(\Omega,\F,P)$, independent
of the $X_{i}^{(N)}$, and suppose further that there is a $D$-tuple
$y_{1},\ldots,y_{D}$ of n.c. random variables in a $C^{*}$-probability
space $(B,\tau_{B})$ with $\tau_{B}$ faithful. If 
\[
M=M(N)=o(N^{\frac{1}{3}})
\]
 and a.s. for any n.c. polynomial $p$ in $D$ variables
\begin{align*}
\|p(Y_{1}^{(M)},\ldots,Y_{D}^{(M)})\| & \to_{M\to\infty}\|p(y_{1},\ldots,y_{D})\|,\\
\tr\left(p(Y_{1}^{(M)},\ldots,Y_{D}^{(M)})\right) & \to_{M\to\infty}\tau_{B}\left(p(y_{1},\ldots,y_{D})\right),
\end{align*}
then a.s. for any n.c. polynomial $p$ in $d+D$ variables, as $N\to\infty$,
\begin{align*}
 & \|p(X_{1}^{(N)}\otimes\id_{M},\ldots,X_{d}^{(N)}\otimes\id_{M},\id_{N}\otimes Y_{1}^{(M)},\ldots,.\id_{N}\otimes Y_{D}^{(M)})\|\to\\
 & \|p(s_{1}\otimes1_{B},\ldots,s_{d}\otimes1_{B},1_{\SS}\otimes y_{1},\ldots,1_{\SS}\otimes y_{D})\|
\end{align*}
and 
\begin{align*}
 & \tr(p(X_{1}^{(N)}\otimes\id_{M},\ldots,X_{d}^{(N)}\otimes\id_{M},\id_{N}\otimes Y_{1}^{(M)},\ldots,.\id_{N}\otimes Y_{D}^{(M)}))\to\\
 & [\tau_{\SS}\otimes\tau_{B}](p(s_{1}\otimes1_{B},\ldots,s_{d}\otimes1_{B},1_{\SS}\otimes y_{1},\ldots,1_{\SS}\otimes y_{D})).
\end{align*}
\end{thm}

This will allow us to deduce the following.
\begin{thm}
\label{thm:first-convergence}There exists a sequence $\{\K^{(i)}\}_{i=1}^{\infty}$
with each $\K^{(i)}(v)\to_{i\to\infty}\infty$ such that a.s. for
any fixed $m\in\N$, for any n.c. polynomial $p$ in $|V|$ variables
\[
\lim_{i\to\infty}\|p(X_{v}^{(m,\K^{(i)})}\,:\,v\in V)\|=\|p(L_{v}^{(m)}+L_{v}^{(m)*}\,:\,v\in V)\|.
\]
The norm on the left is the operator norm w.r.t. the standard Hermitian
norm on $(\C^{m}){}^{\otimes F}\otimes\C^{\otimes\K}$. The norm on
the right is tensor product of operator norm on $(\C^{m}){}^{\otimes F}$
and the $C^{*}$-norm on $\O^{V}(m)$.
\end{thm}

\begin{proof}
To be concrete, let $v_{1},v_{2},\ldots,v_{n}$ be some ordering of
$\V$ and for some $\delta>4$ let
\[
\K^{(i)}(v_{1})\eqdf i;\quad\K^{(i)}(v_{k})\eqdf\K^{(i)}(v_{k-1})^{\delta}\,\quad2\leq k\leq n.
\]
Note $\delta>4$ implies for $2\leq k\leq n$
\begin{equation}
\delta^{k-1}>\frac{3\delta^{k-1}}{(\delta-1)}\geq3(1+\delta+\cdots+\delta^{k-2}).\label{eq:induction}
\end{equation}
For each choice of $\pm$ let 
\[
s_{IJ}^{\pm}\eqdf x_{IJ}^{v\pm}+[x_{IJ}^{v\pm}]^{*}.
\]
Firstly, for any fixed n.c. polynomial $p$ and any fixed $m\in\N$,
from (\ref{eq:X-def}) and (\ref{eq:L-def}) there is a n.c. polynomial
$q=q(\Gamma,p,m)$ with coefficients in $\End((C^{m})^{\otimes F})$
such that 
\begin{align}
 & p(X_{v}^{(m,\K^{(i)})}\,:\,v\in V)\nonumber \\
= & q(X_{IJ}^{v}\otimes\id_{\C_{V\backslash\{v\}}^{\K^{(i)}}},Y_{IJ}^{v}\otimes\id_{\C_{V\backslash\{v\}}^{\K^{(i)}}}),\label{eq:obs1}\\
 & p(L_{v}^{(m)}+L_{v}^{(m)*}\,:\,v\in V)\nonumber \\
= & q(s_{IJ}^{+}\otimes\id_{\bigotimes_{w\neq v}\O_{w}(m)},s_{IJ}^{-}\otimes\id_{\bigotimes_{w\neq v}\O_{w}(m)}).\label{eq:obs2}
\end{align}
The inputs to $q$ above run over $v\in\V$ and $I,J\in[m]^{F(v)}$. 

We make the following observation. For any $W\subset V$, the $C^{*}$-algebra
$(A_{W},\tau_{\vac})$ generated by $s_{IJ}^{+}\otimes\id_{\bigotimes_{w\in W\backslash\{v\}}\O_{w}(m)},s_{IJ}^{-}\otimes\id_{\bigotimes_{w\in W\backslash\{v\}}\O_{w}(m)}$
as $v\in W$ and $I,J\in[m]^{F(v)}$ has the form
\[
(A_{W},\tau_{W})\cong\bigotimes_{v\in W}^{\min}(S_{v},\tau_{v})
\]
 where ($S_{v},\tau_{v})$ is the $C^{*}$-probability space generated
by $2m^{2F(v)}$ free semicircular variables. Each $\tau_{v}$ is
well-known to be faithful (this is a special case of Proposition \ref{prop:main-Cstar}).
Hence by \cite[\S\S 2.3, Prop.]{Avitzour}, $\tau_{W}$ is faithful
on $A_{W}$ (this is also a special case of Proposition \ref{prop:main-Cstar}).

This observation means that one can iterate Theorem \ref{thm:CGP}
--- using (\ref{eq:induction}) --- to obtain:

\textbf{\emph{S: }}\emph{For any fixed $m\in\N$, a.s. for any n.c.
polynomial $q$ with coefficients in $\C$, 
\begin{align*}
 & \lim_{i\to\infty}\|q(X_{IJ}^{v}\otimes\id_{\C_{V\backslash\{v\}}^{\K^{(i)}}},Y_{IJ}^{v}\otimes\id_{\C_{V\backslash\{v\}}^{\K^{(i)}}})\|\\
= & \|q(s_{IJ}^{+}\otimes\id_{\bigotimes_{w\neq v}\O_{w}(m)},s_{IJ}^{-}\otimes\id_{\bigotimes_{w\neq v}\O_{w}(m)})\|,\\
 & \lim_{i\to\infty}\tr\left(q(X_{IJ}^{v}\otimes\id_{\C_{V\backslash\{v\}}^{\K^{(i)}}},Y_{IJ}^{v}\otimes\id_{\C_{V\backslash\{v\}}^{\K^{(i)}}})\right)\\
= & \tau_{_{V}}\left(q(s_{IJ}^{+}\otimes\id_{\bigotimes_{w\neq v}\O_{w}(m)},s_{IJ}^{-}\otimes\id_{\bigotimes_{w\neq v}\O_{w}(m)})\right).
\end{align*}
}Now by a result of Male \cite[Prop. 7.3]{Male}, for $m\in\N$ a.s.
the analogous convergence (mutatis mutandis) holds for any $q$ with
coefficients in $\End((C^{m})^{\otimes F})$. Hence by the observation
around (\ref{eq:obs1}) and (\ref{eq:obs2}), taking the intersection
of these a.s. events over $m\in\N$ gives the result.
\end{proof}
\begin{rem}
In the previous proof, all $\K^{(i)}$ can likely be taken the same
by adapting the results of Belinschi and Capitaine \cite{belinschi2022strong}
to an arbitrary number of tensor products --- only two are dealt
with in \emph{(ibid.)} owing to the intended application therein to
the Peterson-Thom conjecture.
\end{rem}

\section{Second strong limit\label{sec:Second-strong-limit}}

The main result of $\S$\ref{sec:Second-strong-limit} is the following.
\begin{thm}
\label{thm:strong_convergence_technical_theorem}For any complex valued
n.c. $*$-polynomial $p$ in $|V|$ variables and their conjugates,
\[
\lim_{m\to\infty}\|p(L_{v}^{(m)},(L_{v}^{(m)})^{*}:v\in V)\|_{(\C^{m}){}^{\otimes F}\otimes\O^{V}(m)}=\|p(\ell_{v},\ell_{v}^{*}:v\in V)\|_{\O_{\Gamma}}.
\]
\end{thm}

\subsection{Isometries}
\begin{lem}
\label{lem:isometry}For $v\in\V$, $(L_{v}^{(m)})^{*}L_{v}^{(m)}=1$.
\end{lem}

\begin{proof}
We have
\begin{align}
(L_{v}^{(m)})^{*}L_{v}^{(m)} & =\frac{1}{4m^{F(v)}}\id_{(\C^{m}){}^{\otimes F\backslash F(v)}}\otimes\id_{\bigotimes_{w\neq v}\O_{w}(m)}\otimes\nonumber \\
 & \sum_{\substack{I,J,K,L\in[m]^{F(v)}}
}\text{\ensuremath{\left(r_{IJ}^{+}\otimes[x_{IJ}^{v+}]^{*}+r_{IJ}^{-}\otimes[x_{IJ}^{v-}]^{*}\right)\left(r_{KL}^{+}\otimes[x_{KL}^{v+}]+r_{KL}^{-}\otimes[x_{KL}^{v-}]\right)}}\nonumber \\
 & =\frac{1}{4m^{F(v)}}\id_{(\C^{m}){}^{\otimes F\backslash F(v)}}\otimes\id_{\O_{V}(m)}\otimes\sum_{\substack{I,J\in[m]^{F(v)}}
}r_{IJ}^{+}r_{IJ}^{+}+r_{IJ}^{-}r_{IJ}^{-}.\label{eq:temp1}
\end{align}
We calculate
\begin{align*}
r_{IJ}^{+}r_{IJ}^{+}+r_{IJ}^{-}r_{IJ}^{-} & =\left(\e_{IJ}+\e_{JI}\right)\left(\e_{IJ}+\e_{JI}\right)-(\e_{IJ}-\e_{JI})(\e_{IJ}-\e_{JI})\\
 & =2\e_{JI}\e_{IJ}+2\e_{IJ}\e_{JI}=2(\e_{II}+\e_{JJ}).
\end{align*}
Hence
\begin{equation}
\sum_{\substack{I,J\in[m]^{F(v)}}
}r_{IJ}^{+}r_{IJ}^{+}+r_{IJ}^{-}r_{IJ}^{-}=4m^{F(v)}\id_{(\C^{m})^{\otimes F(v)}}.\label{eq:temp2}
\end{equation}
Combining (\ref{eq:temp1}) and (\ref{eq:temp2}) proves the lemma.
\end{proof}

\subsection{T2: A key proposition\label{subsec:T2:-A-key}}

The goal of this $\S\S$\ref{subsec:T2:-A-key} is to prove the following
proposition.
\begin{prop}[Key]
\label{prop:key}For all non-adjacent vertices $v\not\sim w$ in
$\V$,
\[
\lim_{m\to\infty}\|(L_{v}^{(m)})^{*}L_{w}^{(m)}\|_{(\C^{m}){}^{\otimes F}\otimes\O^{V}(m)}=0.
\]
\end{prop}

\begin{proof}
In the following, we view $\End((\C^{m}))^{\otimes F(v)})$ as a subalgebra
of $\End((\C^{m}))^{\otimes F})$ by tensoring with identity (this
makes notation less cumbersome), and similarly for $\End((\C^{m}))^{\otimes F(w)})$.
In the same spirit, for $v\in V$ and $I,J\in[m]^{F}$ we write 
\[
x_{IJ}^{v\pm}\eqdf x_{IJ}^{v\pm}\otimes\id_{\bigotimes_{w\neq v}\O_{w}(m)}.
\]

When we expand $[L_{v}^{*}L_{w}L_{w}^{*}L_{v}]^{p}$ using (\ref{eq:L-def})
we get a sum over
\begin{equation}
\prod_{i=0}^{p-1}r_{I_{i}J_{i}}^{\pm(4i)}\otimes[x_{I_{i}J_{i}}^{v\pm(4i)}]^{*}r_{A_{i}B_{i}}^{\pm(4i+1)}\otimes x_{A_{i}B_{i}}^{w\pm(4i+1)}r_{C_{i}D_{i}}^{\pm(4i+2)}\otimes[x_{C_{i}D_{i}}^{w\pm(4i+2)}]^{*}r_{K_{i}L_{i}}^{\pm(4i+3)}\otimes x_{K_{i}L_{i}}^{v\pm(4i+3)}\label{eq:term-first-expand}
\end{equation}
where $\pm(k)$ refers to a choice of sign $\pm$ depending on $k$.
In this expansion, 
\[
I_{i},J_{i},K_{i},L_{i}\in[m]^{F(v)},\quad A_{i},B_{i},C_{i},D_{i}\in[m]^{F(w)}.
\]

Firstly, for (\ref{eq:term-first-expand}) to be non-zero we must
have
\begin{align*}
I_{i} & =K_{i},\,J_{i}=L_{i}\\
\pm(4i) & =\pm(4i+3)\\
C_{i} & =A_{i+1},\,D_{i}=B_{i+1}\\
\pm(4i+2) & =\pm(4i+5),
\end{align*}
and when these hold (\ref{eq:term-first-expand}) is equal to
\begin{equation}
r_{I_{0}J_{0}}^{\pm(0)}r_{A_{0}B_{0}}^{\pm(1)}r_{C_{0}D_{0}}^{\pm(2)}r_{K_{0}L_{0}}^{\pm(0)}\prod_{i=1}^{p-1}r_{I_{i}J_{i}}^{\pm(4i)}r_{A_{i}B_{i}}^{\pm(4i-2)}r_{C_{i}D_{i}}^{\pm(4i+2)}r_{K_{i}L_{i}}^{\pm(4i)}\otimes x_{A_{0}B_{0}}^{w\pm(1)}[x_{C_{p-1}D_{p-1}}^{w\pm(4p-2)}]^{*}.\label{eq:term}
\end{equation}
When we expand the whole thing out we replace each $r_{IJ}^{\pm}$
with a $\e_{IJ}$ or $\e_{JI}$ up to a scalar. We will now expand
out all $r^{\pm}$ terms other than the second and the second last.

We get after expanding the first and fourth, and summing over the
choice of $\pm(0)$,
\begin{align*}
 & \left(\e_{I_{0}J_{0}}+\e_{J_{0}I_{0}}\right)r_{A_{0}B_{0}}^{\pm(1)}r_{C_{0}D_{0}}^{\pm(2)}\left(\e_{I_{0}J_{0}}+\e_{J_{0}I_{0}}\right)-(\e_{I_{0}J_{0}}-\e_{J_{0}I_{0}})r_{A_{0}B_{0}}^{\pm(1)}r_{C_{0}D_{0}}^{\pm(2)}(\e_{I_{0}J_{0}}-\e_{J_{0}I_{0}})\\
= & 2\e_{I_{0}J_{0}}r_{A_{0}B_{0}}^{\pm(1)}r_{C_{0}D_{0}}^{\pm(2)}\e_{J_{0}I_{0}}+2\e_{J_{0}I_{0}}r_{A_{0}B_{0}}^{\pm(1)}r_{C_{0}D_{0}}^{\pm(2)}\e_{I_{0}J_{0}}
\end{align*}
So if 
\[
X\eqdf\prod_{i=1}^{p-1}r_{I_{i}J_{i}}^{\pm(4i)}r_{A_{i}B_{i}}^{\pm(4i-2)}r_{C_{i}D_{i}}^{\pm(4i+2)}r_{K_{i}L_{i}}^{\pm(4i)}\otimes x_{A_{0}B_{0}}^{w\pm(1)}[x_{C_{p-1}D_{p-1}}^{w\pm(4p-2)}]^{*}
\]
then
\[
\sum_{I_{0},J_{0}\in[m]^{F(v)},\pm(0)}\eqref{eq:term}=4\sum_{I_{0},J_{0}\in[m]^{F(v)}}\e_{I_{0}J_{0}}r_{A_{0}B_{0}}^{\pm(1)}r_{C_{0}D_{0}}^{\pm(2)}\e_{J_{0}I_{0}}X.
\]
Now repeating this argument and putting back in scalar factors we
get
\begin{align}
[L_{v}^{*}L_{w}L_{w}^{*}L_{v}]^{p}=\frac{1}{m^{p|F(v)|}m^{p|F(w)|}2^{4p}}\sum_{\substack{I_{0},J_{0},A_{0},B_{0},C_{0},D_{0},I_{1},J_{1}\ldots,\\
\pm(0),\ldots,\pm(4p-1)
}
}\eqref{eq:term} & =\nonumber \\
\frac{1}{m^{p|F(v)|}m^{p|F(w)|}}\frac{1}{4}\sum_{\substack{I_{0},J_{0},A_{0},B_{0},C_{0},D_{0},I_{1},J_{1},C_{1},D_{1},I_{2},\ldots\\
\pm(1),\pm(4p-2)
}
}\nonumber \\
\e_{I_{0}J_{0}}r_{A_{0}B_{0}}^{\pm(1)}\e_{C_{0}D_{0}}\e_{J_{0}I_{0}}\e_{I_{1}J_{1}}\e_{D_{0}C_{0}}\e_{C_{1}D_{1}}\e_{J_{1}I_{1}}\cdots\nonumber \\
\cdots\e_{I_{p-1}J_{p-1}}\e_{D_{p-2}C_{p-2}}r_{C_{p-1}D_{p-1}}^{\pm(4p-2)}\e_{J_{p-1}I_{p-1}}\otimes x_{A_{0}B_{0}}^{w\pm(1)}[x_{C_{p-1}D_{p-1}}^{w\pm(4p-2)}]^{*}.\label{eq:big-term}
\end{align}

\begin{figure}
\includegraphics[scale=0.2]{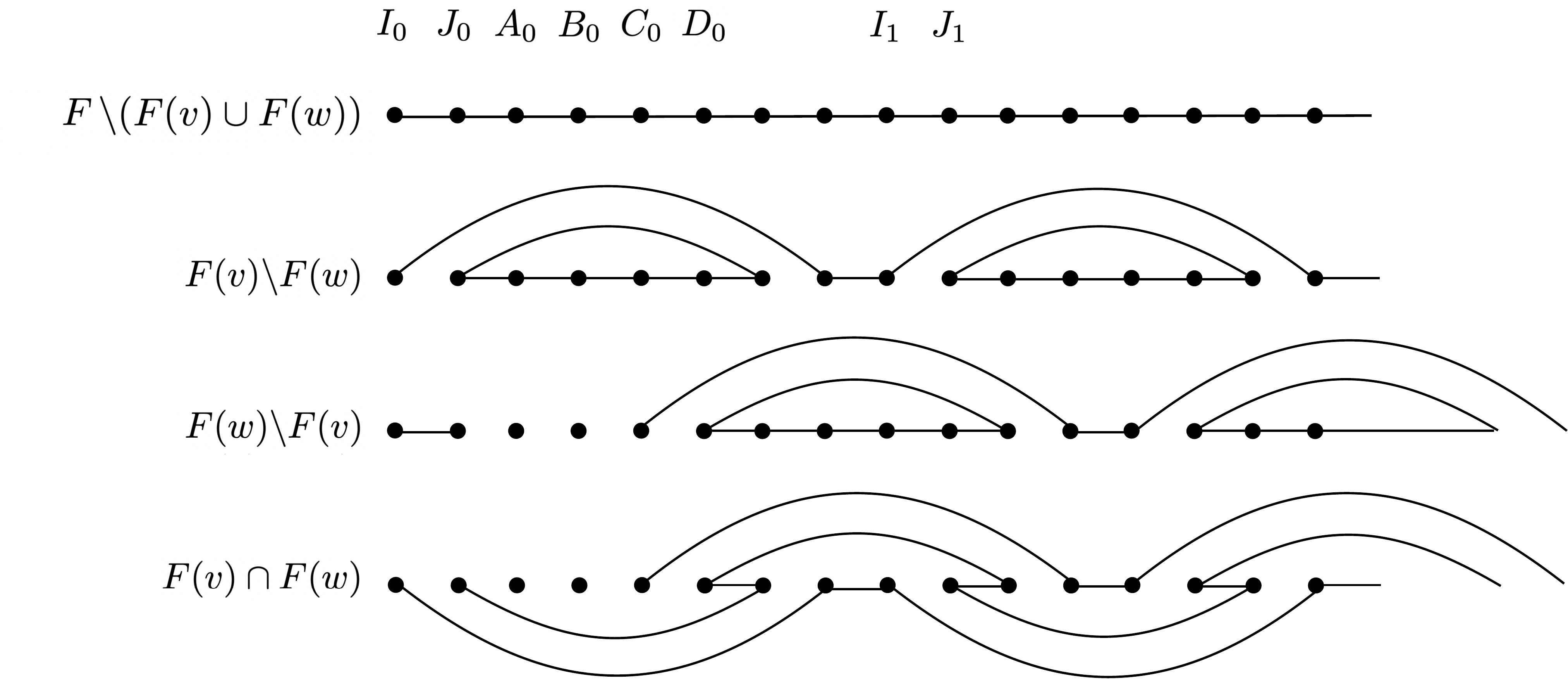}\caption{\label{fig:indices}Constraints on indices in (\ref{eq:big-term})
to yield a non-zero summand. }
\end{figure}

Figure \ref{fig:indices} shows diagrammatically the constraints on
indices of (\ref{eq:big-term}) that are required to make the summand
non-zero. Two indices are joined in the diagram if the corresponding
components of the index must be equal for a non-zero summand for a
given edge, when the edge lives in the set described on the left hand
side of the image. Recall that $I_{i},J_{i},K_{i},L_{i}\in[m]^{F(v)}$
and $A_{i},B_{i},C_{i},D_{i}\in[m]^{F(w)}$ and the corresponding
$\varepsilon$ matrices are extending to the identity on the components
corresponding to edges in $F\setminus F(v)$ and $F\setminus F(w)$
respectively, thus there is no choice to be made for the indices corresponding
to such edges, and so we connect these indices in the diagram to an
existing connected component without joining two existing components.
As such, the number of entries across all indicies that may be freely
chosen is bounded by the number of connected components in the diagram.
There are at most
\[
(6+p)(|F(v)\backslash F(w)|+|F(w)\backslash F(v)|)+6
\]
 connected components and so there are in total at most 
\[
m^{(6+p)(|F(v)\backslash F(w)|+|F(w)\backslash F(v)|)+6}\leq m^{6|F|}m^{p(|F(v)|+|F(w)|-1)}
\]
choices for the indices giving non zero terms since $|F(v)\cap F(w)|=1$
as $v\not\sim w$. There are 4 remaining choices of $\pm(1),\pm(4p-2)$.
Each summand has operator norm $\leq4$ since $\|\e_{IJ}\|=1$, $\|r_{IJ}^{\pm}\|\leq2$,
and $\|x_{A_{0}B_{0}}^{w\pm(1)}\|=\|[x_{C_{p-1}D_{p-1}}^{w\pm(4p-2)}]^{*}\|=1$.

Hence by the triangle inequality
\[
\|\left(L_{v}^{*}L_{w}L_{w}^{*}L_{v}\right)^{p}\|\leq\frac{1}{m^{p|F(v)|}m^{p|F(w)|}}4m^{6|F|}m^{p(|F(v)|+|F(w)|-1)}
\]
and so
\begin{align*}
\|L_{v}^{*}L_{w}L_{w}^{*}L_{v}\| & =\text{\ensuremath{\|\left(L_{v}^{*}L_{w}L_{w}^{*}L_{v}\right)^{p}\|^{\frac{1}{p}}}}\\
 & \leq\ensuremath{\frac{1}{m^{|F(v)|}m^{|F(w)|}}}\ensuremath{\left(4^{1/p}m^{6|F|/p}m^{|F(v)|+|F(w)|-1}\right)}\\
 & \text{\ensuremath{=\frac{1}{m}}}4^{1/p}m^{6|F|/p}.
\end{align*}
Since this holds for any $p\in\N$ it must hold that $\|L_{v}^{*}L_{w}L_{w}^{*}L_{v}\|\leq\frac{1}{m}.$
In any case it tends to zero as $m\to\infty$.
\end{proof}

\subsection{T3: Non-degeneracy }
\begin{lem}
\label{lem:T3}For any $m\in\N$ and $w_{1},\ldots,.w_{k}\in\V$,
$\prod_{i=1}^{k}\left(1-L_{w_{i}}^{(m)}L_{w_{i}}^{(m)*}\right)$ has
a fixed vector.
\end{lem}

\begin{proof}
Let $\{v_{1},\ldots,v_{n}\}=\V$. Let $I$ be an arbitrary index in
$[m]^{F}$, then define $\xi_{m}=e_{I}\otimes\Omega_{v_{1}}\otimes...\otimes\Omega_{v_{n}}\in(\C^{m})^{\otimes F}\otimes\O^{V}(m)$,
where $\Omega_{v}$ is the vacuum vector of $\O_{v}(m)$. Then $\xi_{m}$
is fixed by $I-L_{w}^{(m)}L_{w}^{(m)*}$ for any $w\in\V$ since $L_{w}^{(m)*}$
annihilates $\xi_{m}$. Hence for any $w_{1},...,w_{k}$ in $\V$,
$\xi_{m}$ is fixed by $\prod_{i=1}^{k}\left(1-L_{w_{i}}^{(m)}L_{w_{i}}^{(m)*}\right)$.
\end{proof}

\subsection{Proof of Theorem \ref{thm:strong_convergence_technical_theorem}\label{subsec:Proof-of-Theorem-main}}
\begin{proof}[Proof of Theorem \ref{thm:strong_convergence_technical_theorem}]
Let $\F$ be an arbitrary free ultrafilter on $\N$. Consider the
ultraproduct $C^{*}$-algebra
\[
\U_{\F}\eqdf\prod_{m\to\F}\End((\C^{m}){}^{\otimes F})\otimes\O^{V}(m).
\]
This is the quotient of the product 
\[
\prod_{m\in\N}\End((\C^{m}){}^{\otimes F})\otimes\O^{V}(m)
\]
by the subspace $\mathcal{N}_{\F}$ of bounded sequences $\{a_{m}\}_{_{m\in\N}}$
that tend to zero along $\F$. The $*$-algebraic operations descend
from the product and the norm is given by 
\[
\|\{a_{m}\}_{_{m\in\N}}\|_{\U_{\F}}=\lim_{m\to F}\|a_{m}\|.
\]
We refer the reader to \cite[Appendix A]{BrownOzawa} for background
on ultraproduct $C^{*}$-algebras.

Consider the elements 
\[
\L_{v}\eqdf\{L_{v}^{(m)}\}_{m\in\N}
\]
 of $\U_{\F}$ and let $B_{\Gamma}$ denote the $C^{*}$-subalgebra
generated by the $\L_{v}$. Since by Lemma \ref{lem:isometry} $(L_{v}^{(m)})^{*}L_{v}^{(m)}=1$,
it follows that 
\[
\L_{v}^{*}\L_{v}=1_{\U_{\F}}
\]
and the $\L_{v}$ are isometries.

\textbf{T1: }It is clear that if $v\sim w$ then $L_{v}^{(m)}$ and
$L_{w}^{(m)}$ commute since the elements 
\[
\id_{(\C^{m}){}^{\otimes F\backslash F(v)}}\otimes r_{IJ}^{\pm},\,\id_{(\C^{m}){}^{\otimes F\backslash F(w)}}\otimes r_{IJ}^{\pm}
\]
 in their respective defining sums commute (only changing indices
in different channels) as do the $x_{IJ}^{v\pm}$ (involving different
vertices). Similarly $L_{v}^{(m)}$ and $L_{w}^{(m)*}$ commute. Therefore
we have when $v\sim w$
\begin{align*}
\L_{v}\L_{w} & =\L_{v}\L_{w},\\
\L_{v}^{*}\L_{w} & =\L_{w}\L_{v}^{*}
\end{align*}
as required.

\textbf{T2: }Proposition \ref{prop:key} shows that if $v$ and $w$
are not adjacent in $\Gamma$, then the sequence $\{(L_{v}^{(m)})^{*}\}_{m\in\N}\{L_{w}^{(m)}\}_{m\in\N}\in\mathcal{\mathcal{N}_{\F}}$
and hence in $\U_{\F}$
\[
\L_{v}^{*}\L_{w}=0.
\]

\textbf{T3: }Via pointwise multiplication, $\U_{\F}$ is naturally
a subalgebra of $\B(\H_{\F})$ where $\H_{\F}$ is the ultraproduct
of the Hilbert spaces $(\C^{m})^{\otimes F}\otimes\O^{V}(m)$ along
$\F$. Since for each fixed $w_{1},\ldots,.w_{k}\in\V$, $\prod_{i=1}^{k}\left(1-L_{w_{i}}^{(m)}L_{w_{i}}^{(m)*}\right)$
has a fixed vector $\xi_{m}$ of norm one, the (class of the) sequence
$\{\xi_{m}\}_{m\in\N}$ in $\H_{\F}$ is a norm one vector fixed by
\[
\prod_{i=1}^{k}\left(1-\L_{w_{i}}\L_{w_{i}}^{*}\right),
\]
and so this operator is non-zero.

\textbf{Synthesis. }Therefore by Theorem \ref{thm:universailty} the
assignment
\[
\ell_{v}\mapsto\L_{v}
\]
 extends to a $*$-algebra isomorphism $\psi:\O_{\Gamma}\to\U_{\F}$.
This implies that for any complex valued n.c. $*$-polynomial $p$
in $|V|$ variables and their conjugates,
\[
\lim_{m\to\F}\|p(L_{v}^{(m)},(L_{v}^{(m)})^{*}:v\in V)\|_{(\C^{m}){}^{\otimes F}\otimes\O^{V}(m)}=\|p(\ell_{v},\ell_{v}^{*}:v\in V)\|_{\O_{\Gamma}}.
\]
Since this holds for any free ultrafilter $\F$ on $\N$, and the
right hand side does not depend on $\F$, it must in fact be the case
that Theorem \ref{thm:strong_convergence_technical_theorem} holds
--- i.e. the convergence above holds as a standard limit.
\end{proof}

\section{Proof of Theorem \ref{thm:strong-converge-raags}}

We first prove the following random matrix result that may be of independent
interest.
\begin{thm}
\label{thm:main-random-matrix}There is a sequence $\{(m(n),\K^{(n)})\}_{n=1}^{\infty}$
such that a.s., for any n.c. polynomial $Q$ in $|V|$ variables,
as $n\to\infty$
\[
\|Q(X_{v}^{(m(n),\K^{(n)})}\,:\,v\in V)\|\to\|Q(s_{v}:v\in V)\|.
\]
The norm on the left is the operator norm w.r.t. the standard Hermitian
norm on $\left(\C^{m(n)}\right)^{\otimes F}\otimes\C^{\otimes\mathbf{K}^{(n)}}$,
and the norm on the right is the operator norm on $\mathcal{S}_{\Gamma}$.
\end{thm}

\begin{proof}
Let $\{\K^{(i)}\}_{i=1}^{\infty}$ denote a sequence provided by Theorem
\ref{thm:first-convergence}. Let $\P(V,n)$ denote the space of n.c.
polynomials of degree at most $n$ in $V$-indexed variables. This
can be given the $\ell^{1}$-norm 
\[
\left\Vert \sum_{k\leq n}\sum_{v_{1}\ldots v_{k}}a_{v_{1}v_{2}\cdots v_{k}}X_{v_{1}}X_{v_{2}}\cdots X_{v_{k}}\right\Vert _{1}\eqdf\sum_{k\leq n}\sum_{v_{1}\ldots v_{k}}|a_{v_{1}v_{2}\cdots v_{k}}|.
\]
Note for later that if $x_{v}$ are elements of a Banach algebra and
$p\in\mathcal{P}(V,n)$, then the map 
\[
p\mapsto\|p(x_{v}:v\in V)\|
\]
is Lipschitz w.r.t. $\|.\|_{1}$ with constant
\begin{equation}
\sup_{k\leq n,v_{1}\ldots v_{k}}\|x_{v_{1}}x_{v_{2}}\cdots x_{v_{k}}\|\leq\sup_{k\leq n}\left(\sup_{v\in V}\|x_{v}\|^{k}\right).\label{eq:lip-const}
\end{equation}

Let $S(n)$ be a finite $\frac{1}{4^{n}}$-net for the $\ell^{1}$-unit
ball in $\P(V,n)$ and let $p$ denote some element of $S(n)$. By
Theorem \ref{thm:strong_convergence_technical_theorem} --- applied
to the corresponding unique n.c. $*$-polynomial $q$ such that for
n.c indeterminates $X_{v}$
\[
p(X_{v}+X_{v}^{*}:v\in\V)=q(X_{v},X_{v}^{*}:v\in\V)
\]
--- for any $\epsilon>0$ there is some $m=m(n)\geq n$ such that
\[
\left|\|p(L_{v}^{(m)}+(L_{v}^{(m)})^{*}:v\in V)\|_{(\C^{m}){}^{\otimes F}\otimes\O^{V}(m)}-\|p(s_{v}:v\in V)\|_{\SS_{\Gamma}}\right|<\frac{1}{2n}
\]
for all $p\in S(n)$.

Now by Theorem \ref{thm:first-convergence} --- again applied to
the $q$ --- a.s. there is $i=i(n)\geq n$ such that 
\[
\left|\|p(L_{v}^{(m)}+(L_{v}^{(m)})^{*}:v\in V)\|_{(\C^{m}){}^{\otimes F}\otimes\O^{V}(m)}-\|p(X_{v}^{(m,\K^{(i)})}\,:\,v\in V)\|_{(\C^{m})^{\otimes F}\otimes\C^{\otimes\mathbf{K}}}\right|<\frac{1}{2n}
\]
for all $p\in S(n)$, hence for all $p\in S(n)$
\begin{equation}
\left|\|p(X_{v}^{(m(n),\K^{(i(n))})}\,:\,v\in V)\|-\|p(s_{v}:v\in V)\|\right|<\frac{1}{n}.\label{eq:diff1}
\end{equation}

Having now fixed $m(n)$ and $i(n)$, let $N_{0}\subset\Omega$ be
the event that for infinitely many $n\in\N$ and $v\in V$, $\|X_{v}^{(m(n),\K^{(i(n))})}\|>3$.
Recall that $X_{v}^{(m(n),\K^{(m(n))})}=\id\otimes\tilde{X}_{v}(n)$
where 
\[
\tilde{X}_{v}(n)\in\SGRM(M(v,n),M(v,n)^{-1})
\]
for some $M(v,n)\geq K^{(i(n))}\geq n$. 

In \cite[(5.3),(5.4)]{HaagerupThr} it is proved that
\[
\E[\exp(t\|\tilde{X}_{v}(n)\|)]\leq M(v,n)\exp\left(2t+\frac{t^{2}}{2M(v,n)}\right)
\]
from which it follows from exponential Chebyshev's inequality ---
taking $t=3\log M(v,n)$ --- that for $n$ large enough
\[
\mathbb{P}(\|\tilde{X}_{v}(n)\|>3)\leq\frac{2M(v,n)^{7}}{M(\nu,n)^{9}}\leq2M(v,n)^{-2}\leq\frac{2}{n^{2}}.
\]
Hence by the Borel---Cantelli Lemma $N_{0}$ is a $P$-null set.
By the earlier remarks around (\ref{eq:lip-const}) it follows that
a.s., for all but finitely many $n$, the maps
\begin{align*}
p & \mapsto\|p(X_{v}^{(m(n),\K^{(i(n))})}\,:\,v\in V)\|,\\
p & \mapsto\|p(s_{v}\,:\,v\in V)\|
\end{align*}
are Lipschitz on $\P(V,n)$ w.r.t. $\|.\|_{1}$ with Lipschitz constants
at most $3^{n}$. Combining this with (\ref{eq:diff1}) we get a.s.,
for all but finitely many $n$, for all $p$ in the unit $\ell^{1}$
ball of $\P(V,n)$ 
\[
\left|\|p(X_{v}^{(m(n),\K^{(i(n))})}\,:\,v\in V)\|-\|p(s_{v}:v\in V)\|\right|<\frac{1}{n}+2\left(\frac{3}{4}\right)^{n}.
\]
But \emph{any} fixed n.c. polynomial $Q$ is a fixed scalar multiple
of some such fixed $p$ for all sufficiently large $n$. Hence a.s.,
for any $Q$
\[
\|Q(X_{v}^{(m(n),\K^{(i(n))})}\,:\,v\in V)\|\to\|Q(s_{v}:v\in V)\|
\]
as $n\to\infty$.
\end{proof}
\begin{proof}[Proof of Theorem \ref{thm:strong-converge-raags}]
Let $\{(m(n),K^{(n)})\}_{n=1}^{\infty}$ be as supplied by Theorem
\ref{thm:main-random-matrix} and for $v\in\V$ let
\[
X_{v}(n)\eqdf X_{v}^{(m(n),\K^{(n)})}(\omega)
\]
 for one of the full-measure $\omega\in\Omega$ for which the conclusion
of Theorem \ref{thm:main-random-matrix} holds. 

Let $\F$ be a free ultrafilter on $\N$. Let 
\[
\U_{\F}\eqdf\prod_{n\to\F}\End((\C^{m(n)}){}^{\otimes F})\otimes\End(\C^{\otimes\K^{(n)}}).
\]
The conclusion of Theorem \ref{thm:main-random-matrix} implies that
\[
\theta:s_{v}\mapsto\{X_{v}(n)\}_{n\in\N}
\]
extends to an embedding of $C^{*}$-algebras
\[
\SS_{\Gamma}\hookrightarrow\U_{\F}.
\]
Let $U_{v}(n)\eqdf\psi(X_{v}(n))$ be the result of the continuous
functional calculus applied to $\psi$ and $X_{v}(n)$. These are
unitary matrices since $\psi(\R)\subset S^{1}$. Furthermore from
(\ref{eq:rm00}) we obtain
\[
U_{v}(n)=\psi(\tilde{X}_{v}(n))\otimes\id_{(\C^{m(n)}){}^{\otimes F\backslash F(v)}}\otimes\id_{\C_{V\backslash\{v\}}^{\otimes\K^{(n)}}}
\]
since continuous functional calculus can easily be checked to respect
unital $C^{*}$-algebra embeddings
\[
A\hookrightarrow A\otimes B;\quad a\mapsto a\otimes1_{B}.
\]
Therefore if $v$ and $w$ are non-adjacent in $\Gamma$, $U_{v}(n)$
and $U_{w}(n)$ commute since they act in disjoint channels. Hence
for each $n\in\N$
\[
\lambda(v)\mapsto U_{v}(n)
\]
 extends to a f.d. unitary representation $\pi_{n}:G\Gamma\to\U(N_{n})$
of $G\Gamma$, viewed here as an embedded subgroup of $C_{r}^{*}(G\Gamma)$. 

By Proposition \ref{prop:main-Cstar} the composition
\[
\lambda(v)\mapsto\psi(s_{v})\mapsto\theta(\psi(s_{v}))=\psi(\theta(s_{v}))=\{U_{v}(n)\}_{n\in\N}\eqdf\pi_{\F}(v)
\]
extends to an embedding $C_{\red}^{*}(G\Gamma)\hookrightarrow\U_{\F}$.
Note that statement by itself is not enough to deduce $\lambda(v)\mapsto U_{v}(n)$
extends to a unitary representation --- only an approximate one ---
so the previous discussion about tensor channels was necessary.

Unpacking the statement above, it implies for any $z\in\C[G\Gamma]$
\[
\|\lambda_{G\Gamma}(z)\|=\|\pi_{\F}(z)\|=\lim_{n\to\F}\|\pi_{n}(z)\|.
\]
Since this holds for any free $\F$, it holds that in fact
\[
\lim_{n\to\infty}\|\pi_{n}(z)\|=\|\lambda_{G\Gamma}(z)\|.
\]
\end{proof}

\section{Proof of Theorem \ref{thm:spectral-gap}\label{sec:Proof-of-Theorem-sg}}

We first explain the language of Theorem \ref{thm:spectral-gap}.
Every compact hyperbolic 3-manifold is obtained as 
\[
M=\Lambda\backslash\mathbb{H}^{3}
\]
 where $\Lambda$ is a discrete torsion-free subgroup of $\PSL_{2}(\C)$.
Given any f.d. unitary representation $\pi:\Lambda\to\U(N)$ consider
the fibered product 
\[
E_{\pi}\eqdf\Lambda\backslash_{\pi}[\mathbb{H}^{3}\otimes\C^{N}]
\]
where $\backslash_{\pi}$ means quotient w.r.t. the action $g(z,w)=(g.z,\pi(g).w)$.
This \emph{associated vector bundle $E_{\pi}$ }is a smooth vector
bundle over $M$. Sections of $E_{\pi}$ can be identified with $\C^{N}$-valued
functions on $\PSL_{2}(\C)$ that transform according to 
\begin{align}
f(g.z) & =\pi(g)f(z),\quad\forall g\in\Lambda,\label{eq:Gamma-eq}\\
f(z.k) & =f(z),\quad\forall k\in\mathrm{PSU}_{2}(\C).\label{eq:K-inv}
\end{align}
As such, the Laplacian $\Delta_{\mathbb{H}^{3}}$ acts on such vector-valued
smooth functions coordinate-wise, and can be regarded as a $\pi$-twisted
Laplacian $\Delta_{\pi}$ acting on smooth sections of $E_{\pi}$.
Because this operator arises from $\Delta_{\mathbb{H}^{3}}$, it corresponds
to the quadratic Casimir operator of $\PSL_{2}(\C)$ under the identification
between the universal enveloping algebra of $\mathfrak{psl}_{2}(\C)$
and right-$\PSL_{2}(\C)$-invariant differential operators on smooth
$\C^{N}$-valued functions satisfying (\ref{eq:Gamma-eq}) and (\ref{eq:K-inv}).
It can also be checked that $(\Delta_{\pi}+1)^{-1}$ is compact on
a suitable Sobolev space and hence the spectrum of $\Delta_{\pi}$
consists only of discrete eigenvalues.

As such, using the classification of the unitary dual of $\PSL_{2}(\C)$
and the well-known values of the Casimir operator on spherical vectors
in said irreducible representations, Theorem \ref{thm:spectral-gap}
will follow from the following theorem.
\begin{thm}
\label{thm:rep-theory}Suppose $\Lambda$ is as above and $\pi_{i}:\Lambda\to\U(N_{i})$
are a sequence of f.d. unitary representations that strongly converge
to the regular representation. For any $\eta>0$, for $i$ large enough
depending on $\eta$, no complementary series $\mathcal{C}^{u}$ with
$u\in[\eta,2)$ appears as a sub-representation of the induced representation
$\rho_{i}\eqdf\Ind_{\Lambda}^{\PSL_{2}(\C)}\pi_{i}$. Neither does
the trivial representation appear.
\end{thm}

Recall the following presentation of the complementary series representations
$\mathcal{C}^{u}$ for $0<u<2$ (see for example \cite[Chapter VI.4(b)]{Na59}).
Let $\mathcal{C}^{u}$ be the completion with respect to the inner
product
\[
\left\langle f_{1},f_{2}\right\rangle _{u}=\int_{\C}\int_{\C}\frac{f_{1}(z_{1})\overline{f_{2}(z_{2})}}{|z_{1}-z_{2}|^{-2+u}}\mathrm{d}x_{1}\mathrm{d}x_{2}\mathrm{d}y_{1}\mathrm{d}y_{2},
\]
where $z_{j}=x_{j}+iy_{j}$ for $j=1,2$, of the collection of measurable
functions $f:\C\to\C$ such that 
\[
\int_{\C}\int_{\C}\frac{|f(z_{1})||f(z_{2})|}{|z_{1}-z_{2}|^{-2+u}}\mathrm{d}x_{1}\mathrm{d}x_{2}\mathrm{d}y_{1}\mathrm{d}y_{2}<\infty.
\]
 Define the unitary representation $\rho^{u}$ of $\mathrm{SL}_{2}(\C)$
on $\mathcal{C}^{u}$ by
\[
\rho^{u}(g)f(z)=|bz+d|^{-2-u}f\left(\frac{az+c}{bz+d}\right),
\]
for $g=\begin{pmatrix}a & b\\
c & d
\end{pmatrix}\in\mathrm{SL}_{2}(\C)$. These irreducible representations descend to irreducible representations
on $\PSL_{2}(\C)$, $[g]\mapsto\rho^{u}([g])$ as $\pm I$ lie in
their kernel. A normalized spherical vector for the $\mathrm{PSU}_{2}(\C$)
action of $\rho^{u}$ is given by
\begin{equation}
\psi^{u}(z)=\frac{1}{\sqrt{\pi}}\left(|z|^{2}+1\right)^{-\frac{1}{2}(2+u)}.\label{eq:sph-vector}
\end{equation}
That is, for any $[k]\in\mathrm{PSU}_{2}(\C)$, one has $\rho^{u}([k])\psi^{u}(z)=\psi^{u}(z)$.

To prove Theorem \ref{thm:rep-theory}, we first demonstrate the following
convergence in norm of Fourier transforms with respect to representations.
\begin{prop}
\label{prop:rep-conv}Under the assumptions of Theorem \ref{thm:rep-theory},
for any $f\in C_{c}^{\infty}(\PSL_{2}(\C))$, 
\[
\lim_{i\to\infty}\|\rho_{i}(f)\|=\|\rho_{\infty}(f)\|,
\]
where $\rho_{\infty}$ is the right regular representation of $\PSL_{2}(\C)$
and $\rho_{i}$ are the representations defined in Theorem \ref{thm:rep-theory}.
\end{prop}

\begin{proof}
Again let $\F$ denote a free ultrafilter on $\N$. Let $A_{\pi_{i}}$
denote the $C^{*}$-matrix-subalgebra generated by the image of $\pi_{i}$
. Let $D$ denote a Dirichlet fundamental domain for $\Lambda$ acting
on $\mathbb{H}^{3}$. Since the compact operators $\mathcal{K}\eqdf\mathcal{K}(L^{2}(D))$
are a nuclear $C^{*}$-algebra, they are exact and hence by \cite[\S 9]{HaagerupThr}
--- adapted to ultraproducts in the obvious way --- we have a natural
inclusion
\[
\mathcal{K}\otimes_{\min}\prod_{i\to\F}A_{\pi_{i}}\hookrightarrow\prod_{i\to\F}\mathcal{K}\otimes_{\min}A_{\pi_{i}}.
\]
Hence the embedding $C_{\red}^{*}(\Lambda)\hookrightarrow\prod_{i\to\F}A_{\pi_{i}}$
obtained from strong convergence extends to $\mathcal{K}\otimes_{\min}C_{\red}^{*}(\Lambda)\hookrightarrow\prod_{i\to\F}\mathcal{K}\otimes_{\min}A_{\pi_{i}}.$
Hence for any $z\in\mathcal{K}\otimes\C[\Lambda]$ we have 
\begin{equation}
\lim_{i\to\F}\|[\mathrm{Id}_{\End(\mathcal{K})}\otimes\pi_{i}](z)\|=\|[\mathrm{Id}_{\End(\mathcal{K})}\otimes\lambda_{\Lambda}](z)\|.\label{eq:limit-tensor}
\end{equation}
Again, since the right hand side does not depend on $\F$, the limit
can be replaced by a standard limit.

Now, exactly as in \cite{HM21}, there are compact operators $a_{f}(\lambda)\in\mathcal{K}$
such that there are unitary conjugacies
\begin{align*}
\rho_{i}(f) & \cong\sum_{g\in\Lambda}a_{f}(g)\otimes\pi_{i}(g),\\
\rho_{\infty}(f) & \cong\sum_{g\in\Lambda}a_{f}(g)\otimes\lambda_{\Lambda}(g),
\end{align*}
and since $\Lambda$ is a uniform lattice, each sum above is finitely
supported. Hence applying (\ref{eq:limit-tensor}) with $z=\sum_{g\in\Lambda}a_{f}(g)\otimes g$
now gives the result.
\end{proof}
\begin{proof}[Proof of Theorem \ref{thm:rep-theory}]
 Recall the following Kunze-Stein estimate \cite{St.To78}: let $1\leq p<2$,
then there exists a constant $c_{p}>0$ such that for any $f\in L^{p}(\mathrm{PSL}_{2}(\C))$
that is right-$\mathrm{PSU}_{2}(\C)$-invariant and any $\psi\in L^{2}(\mathrm{PSL}_{2}(\C))$,
\[
\|f*\psi\|_{2}\leq c_{p}\|f\|_{p}\|\psi\|_{2}.
\]
So for any such $f\in C_{c}^{\infty}(\PSL_{2}(\C))$ and $1\leq p<2$,
we have 
\begin{align*}
\|\rho_{\infty}(f)\| & =\sup_{\|\psi\|_{L^{2}(\PSL_{2}(\C))}=1}\|\rho_{\infty}(f)\psi\|_{L^{2}(\PSL_{2}(\C))}\\
 & =\sup_{\|\psi\|_{L^{2}(\PSL_{2}(\C))}=1}\|f*\psi\|_{L^{2}(\PSL_{2}(\C))}\\
 & \leq c_{p}\|f\|_{p}.
\end{align*}
Recall the Cartan decomposition $\mathrm{SL}_{2}(\C)=KA^{+}K$ where
$K=\mathrm{SU}_{2}(\C)$ and 

\[
A^{+}=\left\{ a_{r}=\begin{pmatrix}e^{\frac{r}{2}} & 0\\
0 & e^{-\frac{r}{2}}
\end{pmatrix}:r\geq0\right\} ,
\]
so that the Haar measure on $\mathrm{SL}_{2}(\C$) is $\sinh^{2}(r)\mathrm{d}k\mathrm{d}r\mathrm{d}k'$.
Then given $\varepsilon\ll1$ and $T>1$, let $\tilde{f}_{T}:\mathbb{R}\to\mathbb{R}$
be a smooth approximation to the indicator function $\mathds{1}_{[T,T+1]}$
that is non-negative, equal to $1$ on $[T+\varepsilon,T+1-\varepsilon]$,
and compactly supported in $[T,T+1]$ so that $\tilde{f}_{T}\leq\mathds{1}_{[T,T+1]}$
on $\mathbb{R}$. Now define $f_{T}:\mathrm{SL}_{2}(\mathbb{C})\to\mathbb{R}$
by $f_{T}(ka_{r}k')=\tilde{f}_{T}(r)$. This function descends to
$\PSL_{2}(\C)$ as it is right-$\{\pm I\}$-invariant and we denote
it by the same symbol. Then by characterisation of the Haar measure
on the quotient we have 

\[
\|f_{T}\|_{p}^{p}=\int_{K}\int_{0}^{\infty}\int_{K}|f_{T}(ka_{r}k')|^{p}\sinh^{2}(r)\mathrm{d}k\mathrm{d}r\mathrm{d}k'\leq\int_{T}^{T+1}\sinh^{2}(r)\mathrm{d}r\leq2e^{2T},
\]
resulting in $\|\rho_{\infty}(f_{T})\|\leq C_{p}e^{\frac{2T}{p}}$
for some $C_{p}>0$, for any $1\leq p<2$ and $T$ sufficiently large. 

On the converse, if $\mathcal{C}^{u}$ is a sub-representation of
$\rho_{i}$ for some $u\in[\eta,2)$ then $\|\rho_{i}(f_{T})\|\geq\|\rho^{u}(f_{T})\|=\sup_{\|\psi_{1}\|=\|\psi_{2}\|=1}\left|\left\langle \rho^{u}(f_{T})\psi_{1},\psi_{2}\right\rangle \right|$.
We test this right hand side with $\psi_{1}=\psi_{2}=\psi^{u}$, where
$\psi^{u}$ is as in (\ref{eq:sph-vector}), to obtain a lower bound
for the operator norm. From \cite{Po92,Na59}, it is computed that 

\[
\langle\rho^{u}(a_{r})\psi^{u},\psi^{u}\rangle=\frac{2}{u}\frac{\sinh\left(\frac{1}{2}ur\right)}{\sinh(r)}.
\]
Then, since $\psi^{u}$ is spherical for $\rho^{u}$, we have $\langle\rho^{u}([ka_{r}k'])\psi^{u},\psi^{u}\rangle=\langle\rho^{u}([a_{r}])\psi^{u},\psi^{u}\rangle$
for any $[k],[k']\in\mathrm{PSU}_{2}(\C)$. Thus, since $[g]\mapsto f_{T}([g])\langle\rho^{u}([g])\psi^{u},\psi^{u}\rangle$
lifts to $g\mapsto f_{T}(g)\langle\rho^{u}(g)\psi^{u},\psi^{u}\rangle$
on $\mathrm{SL}_{2}(\C)$ which is right-$\{\pm I\}$-invariant, we
again use characterisation of the Haar measure on the quotient to
compute
\begin{align*}
\left|\left\langle \rho^{u}(f_{T})\psi^{u},\psi^{u}\right\rangle \right| & =\int_{K}\int_{0}^{\infty}\int_{K'}f_{T}(ka_{r}k')\langle\rho^{u}(ka_{r}k')\psi^{u},\psi^{u}\rangle\sinh^{2}(r)\mathrm{d}k\mathrm{d}r\mathrm{d}k'\\
 & =\frac{2}{u}\int_{0}^{\infty}f_{T}(a_{r})\sinh(r)\sinh\left(\frac{1}{2}ur\right)\mathrm{d}r\\
 & \geq\frac{2}{u}\int_{T+\varepsilon}^{T+1-\varepsilon}\sinh(r)\sinh\left(\frac{1}{2}ur\right)\mathrm{d}r\\
 & \geq\frac{1}{u^{2}}e^{T(1+\frac{u}{2})}.
\end{align*}
Now, since by Proposition \ref{prop:rep-conv} $\lim_{i\to\infty}\|\rho_{i}(f_{T})\|=\|\rho_{\infty}(f_{T})\|,$
there exists some $I$ for which $\|\rho_{i}(f_{T})\|\leq\frac{3}{2}\|\rho_{\infty}(f_{T})\|\leq\frac{3}{2}C_{p}e^{\frac{2T}{p}}$
for all $i\geq I$ and all $1\leq p<2$. If $\mathcal{C}^{u}$ is
a sub-representation of $\rho_{i}$ for some $u\in[\eta,2)$ and some
$i\geq I$, then it follows that $\frac{1}{\eta^{2}}e^{T(1+\frac{\eta}{2})}\leq\frac{3}{2}C_{p}e^{\frac{2T}{p}}$
for any $1\leq p<2$. Choosing $p>\frac{2}{1+\frac{\eta}{2}}\in[1,2)$
then gives a contradiction when $T$ is chosen sufficiently large.

For the trivial representation, it is easy to compute that $\|\mathrm{triv}(f_{T})\|\geq e^{2T}$
and so the same contradiction holds.
\end{proof}
\bibliographystyle{amsalpha}
\bibliography{RA}

\noindent Michael Magee, \\
\noindent Department of Mathematical Sciences,\\
Durham University, \\
Lower Mountjoy, DH1 3LE Durham,\\
United Kingdom\\
\\
IAS Princeton,\\
School of Mathematics,\\
1 Einstein Drive,\\
Princeton 08540\\

\noindent \texttt{michael.r.magee@durham.ac.uk}~\linebreak{}

\noindent Joe Thomas, \\
Department of Mathematical Sciences,\\
Durham University, \\
Lower Mountjoy, DH1 3LE Durham,\\
United Kingdom

\noindent \texttt{joe.thomas@durham.ac.uk}\\

\end{document}